\documentclass[a4paper,english,nostix]{birthday}

\usepackage[latin1]{inputenc}
\usepackage{color}
\usepackage{amssymb,esint,graphicx}
\usepackage{amsthm}
\usepackage{amsmath, mathrsfs}
\usepackage{verbatim}
\usepackage{hyperref}
\usepackage{enumerate}
\usepackage{pdfpages}

\newtheorem{theorem}{Theorem}[section]

\newtheorem{lemma}[theorem]{Lemma}
\newtheorem{proposition}[theorem]{Proposition}
\newtheorem{corollary}[theorem]{Corollary}
\newtheorem{definition}[theorem]{Definition}

\newtheorem{remark}[theorem]{Remark}

\newtheorem*{proof*}{Proof}

\numberwithin{equation}{section}

\newcommand{\R}{\ensuremath{\mathbb{R}}}
\newcommand{\N}{\ensuremath{\mathbb{N}}}

\newcommand{\El}{\ensuremath{\mathcal{E}_\lambda}}

\newcommand{\Levy}{\ensuremath{\mathcal{L}}}

\newcommand{\alp}{\alpha}

\newcommand{\kernel}{\ensuremath{\lambda}}

\newcommand{\Hdot}{\dot{H}}

\newcommand{\UC}{\mathscr{UC}}
\newcommand{\EC}{\mathscr{EC}}

\newcommand{\dd}{\mathrm{d}}

\newcommand{\dell}{\partial}
\newcommand{\indikator}{\mathbf{1}_{|z|\leq 1}}

\DeclareMathOperator{\supp}{supp}

\title[On solutions with finite energy]{On the well-posedness of
  solutions with finite energy for nonlocal equations of porous medium type}

\author{F\'elix del Teso, J\o rgen Endal, and Espen R. Jakobsen}
\contact[felix.delteso\@@{}ntnu.no, jorgen.endal\@@{}ntnu.no, espen.jakobsen\@@{}ntnu.no]{Department of Mathematical Sciences\\
Norwegian University of Science and Technology (NTNU)\\
N-7491 Trondheim, Norway}

\begin{document}

\noindent{\em Dedicated to Helge Holden, who never stops inspiring us, on the occasion of his 60th birthday.} 

\begin{abstract} 
We study well-posedness and equivalence of different notions of
solutions with finite energy for nonlocal porous medium type equations
of the form  
$$\dell_tu-A\varphi(u)=0.$$
These equations are possibly degenerate nonlinear diffusion equations
with a general nondecreasing continuous nonlinearity $\varphi$ and the
largest class of linear symmetric nonlocal diffusion operators $A$
considered so far. The operators are defined from a bilinear energy
form $\mathcal{E}$ and may be degenerate and have some $x$-dependence.
The fractional Laplacian, symmetric finite differences, and any
generator of symmetric pure jump L\'evy processes are included. The main
results are (i) an Ole\u{\i}nik type uniqueness result for energy
solutions; (ii) an existence (and uniqueness) result for
distributional solutions with finite energy; and (iii) equivalence between
the two notions of solution, and as a consequence, new well-posedness
results for both notions of solutions. We also obtain quantitative energy and related $L^p$-estimates for distributional solutions. Our uniqueness results are given for a class of
  functions defined from test functions by completion in a certain
  topology. We study rigorously several cases where this space
  coincides with standard function spaces. In particular, for operators
  comparable to fractional Laplacians, we show that this space is a
  parabolic homogeneous fractional Sobolev space.
 
\end{abstract}



\maketitle

\tableofcontents

\section{Introduction}

In this paper we study uniqueness and existence of solutions with
finite energy of the following two related Cauchy problems of  nonlocal porous
medium type, 
\begin{align}\label{EE}
\dell_t u -A^\kernel[\varphi(u)]&=0 &&\text{in}\quad Q_T:=\R^N\times(0,T),\\
u(x,0)&=u_0(x) &&\text{on}\quad \R^N, \label{EIC}\\
\intertext{and}
\label{E}
\dell_t u -\Levy^\mu[\varphi(u)]&=0 &&\text{in}\quad Q_T,\\
u(x,0)&=u_0(x) &&\text{on}\quad \R^N, \label{IC}
\end{align}
where $u=u(x,t)$ is the solution, $T>0$, $A^\lambda$ and $\Levy^\mu$ 
are nonlocal (convection-) diffusion operators, the nonlinearity
$\varphi$ is any continuous nondecreasing function, and $u_0\in L^1\cap
L^\infty$. The problems are nonlinear degenerate parabolic and 
include the fractional porous medium equations \cite{PaQuRoVa12} where
$\Levy^\mu=-(-\Delta)^{\frac\alp2}$ and $\varphi(u)=u|u|^{m-1}$ for
$\alp\in(0,2)$ and $m>0$. Included are also Stefan
problems, filtration equations, and generalized porous medium
equations, see the introductions of \cite{PaQuRoVa12,PaQuRo16, DTEnJa15} for more information. 

Both problems are connected to a bilinear energy form defined as
\begin{equation}\label{energy0}
\mathcal{E}_\kernel[f,g]:=\frac{1}{2}\iint_{\R^N\times\R^N\setminus D}\big(f(y)-f(x)\big)\big(g(y)-g(x)\big)\, \Lambda(\dd x, \dd y),
\end{equation}
where $D:=\{(x,x) : x\in\R^N\}$ is the diagonal and $\Lambda$ is a 
nonnegative Radon measure on $\R^N\times\R^N\setminus D$. The operator
$A^\kernel$ is the generator of $\mathcal{E}_\kernel$ defined by
\begin{align}\label{defgen}
\mathcal{E}_\kernel[f,g]=-\int_{\R^N}f A^\kernel[g]\,\dd x
\end{align}
(see Corollary 1.3.1 in \cite{FuOsTa94}), while $\Levy^\mu=A^\lambda$
for the special case where $\Lambda=\mu(x+\dd y)\dd x$. In general
$A^\kernel$ is symmetric, $x$-dependent, and has no closed expression, while
$\Levy^\mu$ is an $x$-independent operator with integral
representation 
\begin{equation}\label{deflevy}
\Levy ^\mu [\phi](x)=\int_{\R^N\setminus\{0\} }\big( \phi(x+z)-\phi(x)-z\cdot D\phi(x)\indikator \big)\,\mu(\dd z),
\end{equation}
where $D$ is the gradient, $\indikator$ an indicator function, and 
$\mu$ a symmetric (even) nonpositive L\'evy measure satisfying 
$\int|z|^2\wedge 1\,\mu(\dd z)<\infty$. 
The operator $\Levy^\mu$ is nonnegative and symmetric and  the
fractional Laplacian is an example.

A first warning is that $A^\lambda$ is not a
pure diffusion operator in general: Under density and symmetry
assumptions on $\Lambda$, $A^\lambda$ will have an integral
representation like \eqref{deflevy} with $x$-depending $\mu$ plus an
additional drift term! A second warning is that the $x$-dependence in
$A^\lambda$ is restricted, e.g. $-a(x)(-\Delta)^{\frac\alp2}$ is not
covered! We refer to Section 
\ref{sec:assumptions} for precise assumptions and to Section
 \ref{sec:rem} for a discussion and examples of $A^\lambda$. 

The inspiration for this work were the two recent papers \cite{PaQuRo16} and  \cite{DTEnJa15} which contain well-posedness results for energy (or weak)
solutions of \eqref{EE}--\eqref{EIC} and distributional (or very weak) solutions of \eqref{E}--\eqref{IC} respectively. These very general results
requires different 
techniques and formulations. The uniqueness argument of
\cite{DTEnJa15} is based on a complicated resolvent approximation
procedure of Br\'ezis and Crandall \cite{BrCr79}, while in
\cite{PaQuRo16} it is based on an easier and more direct argument
by Ole\u{\i}nik et al. \cite{OlKaCz58}. 

The first part of this paper is devoted to Ole\u{\i}nik type uniqueness
arguments for \eqref{EE}--\eqref{EIC}. We try to push this argument
as far as possible, and in the process we extend some of the results and arguments of \cite{PaQuRo16}. E.g., we remove absolute
continuity, symmetry, and comparability assumptions.  We also discuss
the applicability and
 limitations of the method. 
 Our uniqueness results are given for a class of
  functions defined from test functions by completion in a certain
  topology. We study rigorously several cases where this space
  coincides with standard function spaces. In particular, for operators
  (globally)
  comparable to fractional Laplacians, we show that this space is a
  parabolic homogeneous fractional Sobolev space. In an appendix we
  also provide rigorous definitions and results of these spaces, some
  of which we were not able to find in the literature.

In the second part 
of the paper we study 
the equivalence between energy and distributional formulations in the
setting of \eqref{E}--\eqref{IC}. A main result is a new
existence result for distributional solutions with finite energy. This
existence result and the uniqueness result of \cite{DTEnJa15} is then
transported from distributional solutions to energy solutions by
equivalence, while the Ole\u{\i}nik uniqueness results of the first part is
transported in the other direction. These result are all either new,
or for the Ole\u{\i}nik results, represent a much simpler approach to obtaining
uniqueness compared to \cite{DTEnJa15}. At the end, we give several new quantitative energy and related $L^p$-estimates for distributional solutions.

The type of bilinear form defined in \eqref{energy0}
  plays a central role in probability theory. It is associated with a
  Dirichlet form and a corresponding symmetric Markov process, see
  e.g. \cite{FuOsTa94} for a general theory.  The
  type of ``nonlocal'' bilinear form we consider here is
  similar to those studied in e.g. \cite{ScUe07,BaBaChKa09}. In the
  linear case ($\varphi(u)=u$), equations \eqref{EE} and \eqref{E} are
  (at least formally) Kolmogorov equations for the transition
  probability densities of the corresponding Markov processes (see
  e.g. Section 3.5.3 in \cite{App09}). 

Let us
  now give a brief summary of previous works on
  \eqref{EE}--\eqref{EIC} and \eqref{E}--\eqref{IC}. We focus first on
  the $x$-dependent equation \eqref{EE}. 
 In the linear case there is a large amount of literature. Some of the main trends
 in the more PDE oriented community are described in the two surveys
 \cite{KaSc14, Sil14} (along with extensions to other types of nonlinear
 equations). When $\varphi$ is nonlinear, we are not aware of any other
  result than the ones presented in \cite{PaQuRo16}. There the
authors consider operators $A^\kernel$ where the densities of the
measures are comparable to the density of the fractional
Laplacian. Existence and uniqueness is discussed in the first part,
but the main focus of the paper is to prove continuity/regularity and
long time asymptotics for energy solutions.

There is a vast literature on special cases of
\eqref{E}--\eqref{IC}. In the linear fractional case
$\dell_tu+(-\triangle)^{\frac{\alpha}{2}}u=0$ for $\alpha\in(0,2)$, we
have well-posedness even for measure data and solutions growing at
infinity \cite{BaPeSoVa14, BoSiVa16}. If we replace
$(-\triangle)^{\frac{\alpha}{2}}$ by an operator $\Levy$ whose measure
has integrable density, well-posedness results can be found in
\cite{BrPa15}. In the case of the fractional porous medium equation
(see above), existence, uniqueness and a priori estimates are proven for
(strong) $L^1$-energy solutions in \cite{PaQuRoVa11, PaQuRoVa12}. 
We also mention that there are results for that equation in weighted
$L^1$-spaces \cite{BoVa14}, with logarithmic diffusion
($\varphi(u)=\log(1+u)$) \cite{PaQuRoVa14}, singular or ultra fast
diffusions \cite{BoSeVa16}, weighted equations with measure data \cite{GrMuPu15}, and problems on bounded domains \cite{BoSiVa14,
  BoVa15, BoVa16}. There are other ways to investigate these
equations: In \cite{BiKaMo10,CaVa11,StanTesoVazCRAS, BiImKa15, 
  StTeVa16}, the authors consider a so-called porous medium equation
with fractional pressure, and in \cite{AMRT10} they consider bounded
diffusion operators that can be represented by nonsingular integral
operators on the form \eqref{deflevy}. Finally, we mention that in the
presence of (nonlinear) convection in \eqref{E}--\eqref{IC}, additional
entropy conditions are needed to have uniqueness \cite{Ali07, CiJa11,
  EnJa14}; a counterexample for uniqueness of distributional solutions
is given in \cite{AlAn10}.

\subsection*{Outline} In Section \ref{sec:mainresults} we state the
assumptions and present and discuss our main results. The main
uniqueness result is proven in Section
\ref{sec:proofofgeneraluniquenessx}. Properties such as equivalence of
distributional and energy solutions, existence of distributional
solutions with finite energy, and energy and $L^p$-estimates are
finally proven in Section \ref{sec:propdistsoln}. In
  Appendices \ref{sec:spaceX}, \ref{sec:homfracSob}, and
  \ref{sec:proofcomparablefracLap} we give rigorous results on the
   Sobolev spaces we use in this paper along
  with the proofs of characterizations of the uniqueness function
  class in terms of common function spaces.

\subsection*{Notation} 
We use the same notation as in \cite{DTEnJa15} except for the ones we explicitly mention here:
The (Borel) measure $\mu$ is said to be {\em even} if $\mu(B)=\mu(-B)$ for all Borel sets $B$. 
We say that the (Borel) measure $\Lambda(\dd x, \dd y)$ is {\em
  symmetric} if $\Lambda(\dd x, \dd y)=\Lambda(\dd y, \dd x)$. A {\em
  kernel} $\kernel(x,\dd y)$ on $\R^N\times
\mathcal{B}(\R^N\setminus\{x\})$ satisfies: (i) $B\mapsto \kernel(x,B)$
is a positive measure on $\mathcal{B}(\R^N\setminus\{x\})$ for each
fixed $x\in\R^N$; and (ii) $x\mapsto\kernel(x,B)$ is a Borel measurable
function for every $B\in \mathcal{B}(\R^N\setminus\{x\})$. An operator
$L$ is {\em symmetric} on $L^2$ if $(u, Lv)_{L^2}=(Lu,v)_{L^2}$.
From the {\em bilinear form} $\mathcal{E}_\kernel$ defined in \eqref{energy0} we define a seminorm (the energy) and
a space,
\begin{align*}
&|f|_{E_\kernel}^2:=\mathcal{E}_\kernel[f,f]
\quad\text{and}\quad E_{\kernel}(\R^N):=\big\{f\textup{ is measurable}\,:\,
|f|_{E_\kernel}<\infty\big\},
\end{align*}
and the related parabolic (energy) seminorm and space,
\begin{align*}
&|f|_{T,E_\lambda}^2:=\int_0^T|f(\cdot,t)|_{E_\lambda}^2\,\dd t,\\
&L^2(0,T;E_{\kernel}(\R^N)):=\big\{f:Q_T\to\R \textup{ is measurable}\,:\,
  |f|_{T,E_\lambda}<\infty\big\}.
\end{align*}
The Cauchy-Schwartz inequality holds in this setting (cf. Lemma \ref{C-S}):
\begin{align*}\label{CS-ineq}
\left|\int_0^T \mathcal{E}_\kernel[f(\cdot,t),g(\cdot,t)]\,\dd t\right|\leq
|f|_{T,E_\lambda} |g|_{T,E_\lambda}.
\end{align*}


\section{Main results}
\label{sec:mainresults}

In this section we give the assumptions, main results, and a
discussion of these. There are two sections with results. Section
\ref{sec:uniquenessenergysolns} contains a sequence of uniqueness
results for energy solutions of 
\eqref{EE}--\eqref{EIC}, while Section \ref{sec:equiv} contains results about
\eqref{E}--\eqref{IC}. There we prove the equivalence of energy and
distributional solutions with finite energy, the existence of the
latter type of solutions, and transport uniqueness and existence
results between the two formulations. The results we obtain are either
new or represent a much more efficient way to obtain such results
compared to previous arguments.

\subsection{Assumptions}
\label{sec:assumptions}

We start by the bilinear form $\mathcal{E}_\kernel$ defined in
\eqref{energy0}. To have a more practical formulation of the assumptions, we first
rewrite \eqref{energy0}: We assume that $\Lambda$ has as kernel $\tilde{\kernel}\geq0$ with
respect to $\dd x$, $\Lambda(\dd x, \dd y)=\tilde{\kernel}(x,\dd y)\dd x,$
change variables $y\to x+z$, and set
$\lambda(x,\dd z):=\tilde\lambda(x,x+\dd z)$ to obtain
\begin{equation}\label{energy}
\mathcal{E}_\kernel[f,g]=\frac{1}{2}\int_{\R^N}\int_{|z|>0}\big(f(x+z)-f(x)\big)\big(g(x+z)-g(x)\big)\, \kernel(x,\dd z)\dd x.
\end{equation}
Our assumptions on $\mathcal{E}_\kernel$ can then be formulated as follows:
\begin{itemize}
\item[\bf (A$_{\kernel 0}$)] $\Lambda$ has as kernel $\tilde\kernel\geq0$ on
$\R^N\times\mathcal{B}(\R^N\setminus\{x\})$,
$$\Lambda(\dd x, \dd y)=\tilde\kernel(x,\dd y)\dd x.$$
\item[\bf (A$_{\kernel 1}$)] The translated kernel $\lambda(x,\dd
  z):=\tilde\lambda(x,x+\dd z)$ satisfies 
\item[] \ (i) \ $\Sigma_\kernel(x):=\int_{0<|z|\leq
    1}|z|^2\kernel(x,\dd z)\in L_\textup{loc}^1(\R^N)$; and
\item[]  \ (ii) \ $\Pi_\kernel(x):=\int_{|z|>1}\kernel(x,\dd z)\in
  L_\textup{loc}^1(\R^N)$. 
\item[\bf (A$_{\kernel 2}$)]  $\Lambda$ is symmetric, 
$$\displaystyle\int_A\int_B\Lambda(\dd x, \dd
y)=\int_B\int_A\Lambda(\dd x, \dd y)\quad\text{for all Borel}\quad A\times B\subset \R^N\times\R^N\setminus D.$$ 
\end{itemize}
In some results, we need to strengthen assumption {\bf (A$_{\kernel 1}$)}.
\begin{itemize}
\item[\bf (A$_{\kernel 1}$')] Assumption {\bf (A$_{\kernel 1}$)} holds and in
  addition
\item[] \ (i) \ $\Pi_\lambda\in L^\infty(\R^N)$; and 
\item[] \ (ii) \ $\lambda(x,\dd z)$ is locally shift-bounded: For
  some constant $C>0$,
$$
\lambda(x+h,B)\leq C \lambda(x,B)\quad\text{for all}\quad
x,h\in\R^N,\ |h|\leq1,\  \text{Borel }B\subset B(0,1)\setminus\{0\}.
$$
\vspace{-0.6cm}
\item[\bf (A$_{\kernel 1}$'')] Assumption {\bf (A$_{\kernel 1}$)} holds and in
  addition 
$$m\mu_\alp(\dd z)\leq \lambda(x,\dd z)\leq M\mu_\alp(\dd z)\qquad\text{where}\qquad
\mu_\alp(\dd z)=\frac{c_{N,\alpha}\dd z}{|z|^{N+\alp}},$$
for some $0<m\leq M$, $\alp\in(0,2)$, and every $x\in \R^N$.
\end{itemize}
The remaining assumptions we will use in this paper are given below.
\begin{itemize}
\smallskip
\item[\bf($\textup{A}_{\mu}$)]
$\mu\geq0$ is an even Radon measure on $\R^N\setminus\{0\}$ satisfying
$$\int_{|z|\leq1}|z|^2\,\mu(\dd z)+\int_{|z|>1}1\,\mu(\dd z)<\infty.$$
\item[\bf($\textup{A}_\varphi$)] $\varphi:\R\to\R$ is continuous and
    nondecreasing.
\smallskip
\item[\bf($\textup{A}_{u_0}$)] $u_0\in L^{1}(\mathbb{R}^N)\cap
  L^\infty(\R^N)$.
\end{itemize}

\begin{remark}\label{consequencesassumptions}
\begin{enumerate}[(a)]
\item By \textup{(A$_{\kernel 0}$)} and \textup{(A$_{\kernel 1}$)},
  $\mathcal{E}_\kernel[\cdot,\cdot]$ is well-defined on
  $C_{\textup{c}}^\infty(\R^N)$, nonnegative and symmetric,  
$$\mathcal{E}_\kernel[f,f]\geq
0\quad\text{and}\quad\mathcal{E}_\kernel[f,g]=\mathcal{E}_\kernel[g,f]\quad\text{for}\quad
f,g \in C_{\textup{c}}^\infty(\R^N).$$
Moreover, by Example 1.2.4 in \cite{FuOsTa94}, $(\mathcal{E}_\kernel,
C_{\textup{c}}^\infty(\R^N))$ is a closable Markovian form on
$L^2(\R^N)$ and its closure a regular Dirichlet form. 
\item It is easy to check that \textup{(A$_{\kernel 1}$'')} 
$\Rightarrow$ \textup{(A$_{\kernel1}$')} 
$\Rightarrow$ \textup{(A$_{\kernel 1}$)}, see also the
  remarks on locally shift-bounded kernels in Section
  \ref{sec:rem}. Assumption \textup{(A$_{\kernel 1}$'')} implies that
  $A^\kernel$ is comparable to $-(-\Delta)^{\frac\alp2}$, and local
  shift-boundedness in \textup{(A$_{\kernel 1}$')} is used to show
  that functions with finite energy can be approximated by test
  functions (cf. Theorem \ref{prop:characterizationofX}). 
\item By \textup{($\textup{A}_{\mu}$)}, the operator $\Levy^\mu$ defined by
  \eqref{deflevy} is well-defined on $C^2(\R^N)\cap L^\infty(\R^N)$,
  nonpositive and symmetric. The generator of any symmetric pure jump
  L\'evy process is included, like e.g. the fractional Laplacian and
  symmetric finite difference operators.
\item If $\kernel(x,\dd z)=\mu(\dd z)$, then  
$$\text{\textup{($\textup{A}_{\mu}$)}}
\qquad\Longrightarrow\qquad\text{\textup{(A$_{\lambda 0}$)}},\quad\text{\textup{(A$_{\lambda 1}$')}, \  and \  \textup{(A$_{\lambda 2}$)}}.$$
The first two trivially hold, while \textup{(A$_{\lambda
  2}$)} holds by e.g. Lemma
6.4 in \cite{DTEnJa15}. 
\item Without loss of generality we can assume $\varphi(0)=0$ (by adding a constant).
\end{enumerate}
\end{remark}

\subsection{Uniqueness results for energy solutions}
\label{sec:uniquenessenergysolns}

In this section we give several uniqueness results for energy (or weak)
solutions of \eqref{EE}--\eqref{EIC}. These results follow
from an extension of the Ole\u{\i}nik argument. 

\begin{definition}[Energy solutions]\label{energysoln}
A function $u\in
L_\textup{loc}^1(Q_T)$ is an energy solution of \eqref{EE}--\eqref{EIC} if
\begin{enumerate}[(i)]
\item $\varphi(u)\in L^2(0,T;E_{\kernel}(\R^N))$; and
\item for all $\psi\in C_\textup{c}^\infty(\R^N\times[0,T)),$
$$\int_{0}^T\left(\int_{\R^N}u\dell_t\psi\,\dd
    x-\mathcal{E}_\kernel[\varphi(u),\psi] \right)\dd t+
  \int_{\R^N}u_0(x)\psi(x,0)\,\dd x=0.$$
\end{enumerate}
\end{definition}

\begin{remark}
\begin{enumerate}[(a)]
\item The integrals in (ii) are well-defined by
\text{\textup{(A$_{\lambda 0}$)}}, \textup{(A$_{\kernel1}$')}, \textup{($\textup{A}_{u_0}$)}, and the 
regularity of $u$ and $\varphi(u)$. From (ii) it follows that the
initial condition $u_0$ is assumed in the distributional sense ($u_0$
is a weak initial trace): 
$$\underset{t\to0^+}{\textup{ess\,lim}}\int_{\R^N}u(x,t)\psi(x,t)\dd
x=\int_{\R^N}u_0(x)\psi(x,0)\dd x\quad\forall\psi\in
C_\textup{c}^\infty(\R^N\times[0,T)).$$ 
\item By the support of the test functions, we could take $L^2_{\mathrm{loc}}([0,T);E_{\kernel}(\R^N))$~in~(i).
\end{enumerate}
\end{remark}

To state the uniqueness results, we will introduce spaces in which the
Ole\u{\i}nik argument works. 
A particular requirement is that test functions are dense in these
spaces w.r.t. to the weakest convergence that can be used in the
proof. This is encoded in the following space:
\begin{align*}
X:=&\Big\{f\in L^\infty(Q_T)\cap L^2(0,T;E_\lambda(\R^N)):\\
&\quad\text{there exists }\{\psi_n\}_{n\in\N}\subset
C_{\textup{c}}^\infty(\R^N\times [0,T)) \text{ such that }\\
&\quad|\psi_n-f|_{T,E_\lambda}\to 0\quad\text{as}\quad \text{$n\to\infty$, and}\\
&\quad\iint_{Q_T}\psi_n \phi\,\dd x\dd t\to \iint_{Q_T}f \phi
 \,\dd x\dd t\quad\text{for all}\quad\phi\in L^1(Q_T)\quad\text{as}\quad n\to\infty\Big\}.
\end{align*}
Below we show that  
 limits can be avoided to get more useful
characterizations of such spaces if we (i) go to subspaces, e.g. 
\begin{align}\label{subsetX}
 X\cap L^2(Q_T)=L^2(Q_T)\cap L^\infty(Q_T)\cap L^2(0,T;E_\lambda(\R^N));
\end{align}
or (ii) restrict the operator by assuming
\textup{(A$_{\lambda 1}$'')} which implies 
\begin{align}\label{alphaX}X=L^\infty(Q_T)\cap L^2(0,T;E_{\mu_\alpha}(\R^N)).
\end{align}
We refer to Theorem \ref{prop:characterizationofX} below for precise statements.

Our most general uniqueness result applies to energy solutions in the
following class of functions:
\begin{equation*}
\begin{split}
\UC:=&\Big\{u \in L^1(Q_T)\cap
L^\infty(Q_T) : \varphi(u)\in  X\Big\}.
\end{split}
\end{equation*}

\begin{theorem}[Uniqueness 1]\label{generaluniquenessx}
Assume \textup{($\textup{A}_\varphi$)}, \text{\textup{(A$_{\lambda
      0}$)}}, \textup{(A$_{\lambda 1}$)}, and
\textup{($\textup{A}_{u_0}$)}. Then there is 
at most one {\em energy} solution $u$ 
of \eqref{EE}--\eqref{EIC} in $\UC$.
\end{theorem}

A proof can be found in Section
\ref{sec:proofofgeneraluniquenessx}.

\begin{remark}
  A similar but less general uniqueness result is given by Theorem 1.1 in \cite{PaQuRo16}. They assume that $\lambda(x,\dd z)$ is absolutely
  continuous with a density comparable to the L\'evy
  measure of the $\alp$-stable process, and hence $A^\lambda$ is
  comparable to $-(-\Delta)^{\frac\alp2}$. In this case
  \textup{(A$_{\kernel1}$')} is satisfied
  in view of the discussion in
 Section \ref{sec:rem}.
\end{remark}
 
Note that in general the uniqueness class $\UC$ is smaller than the
natural existence class 
\begin{equation*}
\EC:=\Big\{u \in L^1(Q_T)\cap
L^\infty(Q_T) : \varphi(u)\in L^2(0,T;E_\kernel(\R^N))\Big\}.
\end{equation*}
This is an intrinsic problem with the Ole\u{\i}nik argument when it is
extended to such general settings as we consider here, and it is also
observed in \cite{PaQuRo16}. 
However, the two classes may coincide under additional assumptions,
e.g. if $\varphi(u)$ also belongs to $L^2\cap L^\infty$  or if
$A^\kernel$ is comparable to $-(-\Delta)^{\frac\alp2}$. This is a
consequence of the following result.

\begin{theorem}\label{prop:characterizationofX} Assume
  \text{\textup{(A$_{\lambda 0}$)}} and \textup{(A$_{\lambda 2}$)}.
\begin{enumerate}[(a)]
\item If \textup{(A$_{\lambda 1}$')} holds, then \eqref{subsetX} holds.
\item If \textup{(A$_{\lambda 1}$'')}  holds, then \eqref{alphaX} holds.
\end{enumerate}
\end{theorem}

The proofs are given in Appendices \ref{sec:spaceX} and
\ref{sec:proofcomparablefracLap} respectively. See also Section \ref{sec:rem} for a possible
alternative based on recurrence. By Theorem
\ref{generaluniquenessx} and Theorem \ref{prop:characterizationofX}, we now have:

\begin{corollary}[Uniqueness 2] 
Assume \textup{(A$_{\lambda0}$)}, \textup{(A$_{\lambda 2}$)}, \textup{($\textup{A}_\varphi$)}, and
  \textup{($\textup{A}_{u_0}$)} hold.
\begin{enumerate}[(a)]
\item If \textup{(A$_{\lambda 1}$')} holds, then there is at most one {\em energy} solution  $u$ of
  \eqref{EE}--\eqref{EIC} such that $u\in\EC$ and $\varphi(u)\in L^2(Q_T)$.
\item  If \textup{(A$_{\lambda 1}$'')} holds, then there
is at most one {\em energy} 
  solution  $u$ of \eqref{EE}--\eqref{EIC} such that $u\in\EC$.
\end{enumerate}
\end{corollary}

\begin{remark} When the operator $A^\lambda$ is comparable to the
  fractional Laplacian $-(-\Delta)^{\frac{\alpha}{2}}$ for
  $\alpha\in(0,2)$ (i.e. \textup{(A$_{\lambda 1}$'')} holds), 
the  uniqueness and existence classes coincide, and if $N>\alpha$ they satisfy
\begin{equation}\label{eq:exun}
\UC=\EC=\Big\{u \in L^1(Q_T)\cap
L^\infty(Q_T) : \varphi(u)\in L^2(0,T;\Hdot^{\frac{\alpha}{2}}(\R^N))\Big\}.
\end{equation}
The latter space is often used in the porous medium setting
\cite{Vaz07,PaQuRoVa12}, see also \cite{PaQuRo16}.
See Appendix \ref{sec:homfracSob} for rigorous definitions and
properties of the homogeneous fractional Sobolev spaces
$\Hdot^{\frac{\alpha}{2}}(\R^N)$ and
$L^2(0,T;\Hdot^{\frac{\alpha}{2}}(\R^N))$, some of these we were not able to find
in the literature.
\end{remark}

Note that if \textup{($\textup{A}_\varphi$)} holds
and $u\in L^\infty(Q_T)$, then $\varphi(u)\in L^\infty(Q_T)$. 
Now let $\beta\in(0,1]$ and assume $\varphi$ is locally $\beta$-H\"older continuous at $0$:
\begin{align}
\label{holder}
\sup_{|s|<R}\frac{|\varphi(s)-\varphi(0)|}{|s|^\beta}<\infty\qquad\text{for
all}\qquad
R>0.
\end{align}
Then, since $\varphi(0)=0$ and $u\in L^\infty(Q_T)$, 
$$ u\in L^{2\beta}(Q_T)\qquad\Longrightarrow\qquad \varphi(u)\in L^2(Q_T).$$
By interpolation, functions $u\in L^1\cap L^\infty$ belongs to $L^{2\beta}$ for
$\beta\in[\frac12,1]$. 
This leads us to our next result:

\begin{corollary}[Uniqueness 3]  Assume
  \textup{($\textup{A}_\varphi$)}, \text{\textup{(A$_{\lambda 0}$)}}, 
  \textup{(A$_{\lambda 1}$')}, \textup{(A$_{\lambda 2}$)}, and
  \textup{($\textup{A}_{u_0}$)} hold. If in addition \eqref{holder} holds for some
  $\beta\in[\frac12,1]$, then there is at most one {\em energy} solution
  $u$ of \eqref{EE}--\eqref{EIC} such that $u\in \EC$. 
\end{corollary}

Now we specialize to the case $\lambda(x,\dd z)=\mu(\dd z)$ and  $A^\kernel
= \Levy^\mu$. Equation \eqref{EE} then becomes equation \eqref{E}. From all the
above uniqueness results and Remark \ref{consequencesassumptions} (d) we obtain
the following uniqueness results for \eqref{E}--\eqref{IC}. 

\begin{corollary}[Uniqueness 4]\label{uniqueness}
Assume \textup{($\textup{A}_\varphi$)}, \textup{($\textup{A}_{\mu}$)}, and
\textup{($\textup{A}_{u_0}$)} hold.
\smallskip

\noindent(a) There is at most one {\em energy} solution $u$
of \eqref{E}--\eqref{IC} in $\UC$.

\smallskip
\noindent(b) There is at most one {\em energy} solution $u$
of \eqref{E}--\eqref{IC} such that $u\in\EC$ and $\varphi(u)\in
L^2(Q_T)$.
\smallskip

\noindent(c) If in addition \eqref{holder} holds for some
  $\beta\in[\frac12,1]$, then there is at most one {\em energy} solution
  $u$ of \eqref{E}--\eqref{IC} such that $u\in \EC$.
\end{corollary}

\subsection{Equivalence with distributional
  solutions and consequences}
\label{sec:equiv}
In this section we study the connection between distributional (or very weak)
solutions and energy (or weak) solutions. We focus on the
simpler case where $A^\kernel=\Levy^{\mu}$, and hence the measure
$\kernel(x,\dd z)=\mu(\dd z)$ is independent of $x$. In other words, we
consider the Cauchy problem \eqref{E}--\eqref{IC}.  In general,
$A^\kernel$ will have an additional drift/convection term compared to $\Levy^\mu$, see Section
\ref{sec:rem}. This gives rise to a nonlinear convection term
in the equation and the possibility that solutions develop shocks
(cf. e.g. \cite{CiJa11} and references therein). Whether this
happens or not here is not known and another reason to avoid
this  case now. 

We state an equivalence result for the two
solution concepts,  existence and uniqueness results for
distributional solutions with finite energy, and
then transport these results from distributional solutions to energy
solutions. The uniqueness results of the previous section are
transported in the opposite direction, and the different uniqueness
results are then compared. We also give quantitative
energy and related $L^p$-estimates for distributional solutions.

\begin{definition}[Distributional solutions]\label{distsol} 
A function $u\in L_\textup{loc}^1(Q_T)$ is a distributional solution
of \eqref{E}--\eqref{IC} if
$$\displaystyle\iint_{Q_T}\Big(u\dell_t \psi+\varphi(u)\Levy^\mu[\psi]\Big)\,\dd x\dd t+
  \int_{\R^N}u_0(x)\psi(x,0)\,\dd x=0 \text{ }\forall\psi\in
  C_\textup{c}^\infty(\R^N\times[0,T)).$$
\end{definition}
The integral is well-defined under the assumptions
\textup{($\textup{A}_\varphi$)}, \textup{($\textup{A}_{\mu}$)}, and \textup{($\textup{A}_{u_0}$)}
if also $\varphi(u)\in L^\infty$ (which is the case when $u\in
L^\infty$).
 This weaker notion of solutions does 
not require finite energy, but when the energy is finite, the two
notions of solutions will be equivalent.

\begin{theorem}[Equivalent notions of solutions]\label{equivalencedistrenergy}
Assume \textup{($\textup{A}_\varphi$)}, \textup{($\textup{A}_{\mu}$)}, \linebreak $u_0\in
L^\infty(\R^N)$, and $u\in L^\infty(Q_T)$. Then
the following statements are equivalent: 
\begin{enumerate}[(a)]
\item $u$ is an {\em energy} solution of \eqref{E}--\eqref{IC}.
\item $u$ is a {\em distributional} solution of \eqref{E}--\eqref{IC} such that 
$$\varphi(u)\in L^2(0,T;E_\mu(\R^N)).$$
\end{enumerate}
\end{theorem}

We prove this result in Section \ref{sec:equivnotionofsolutions}. In the setting of this paper, it turns out that there always exists
distributional solutions with finite energy.
\begin{theorem}[Existence 1]\label{ex_Dsol}
Assume \textup{($\textup{A}_\varphi$)}, \textup{($\textup{A}_{\mu}$)}, and
\textup{($\textup{A}_{u_0}$)}. Then there exists a \textup{distributional} solution $u$
of \eqref{E}--\eqref{IC} satisfying  
\begin{enumerate}[(i)]
\item $u\in L^1(Q_T)\cap L^\infty(Q_T)\cap 
C([0,T];L_\textup{loc}^1(\R^N))$; and
\item $\varphi(u)\in
L^2(0,T;E_\mu(\R^N)).$
\end{enumerate}
\end{theorem}

This is one of the main results of this paper and will be proven at the end of Section \ref{sec:approxEconv}. For such solutions we have a
new uniqueness result by  
equivalence, Theorem \ref{equivalencedistrenergy}, and the uniqueness
result for energy solutions in Corollary \ref{uniqueness}.
\begin{corollary}[Uniqueness 5]\label{uniqueness_d}
Assume \textup{($\textup{A}_\varphi$)}, \textup{($\textup{A}_{\mu}$)}, and
\textup{($\textup{A}_{u_0}$)} hold.
\smallskip

\noindent(a) There is at most one {\em distributional} solution $u$ of \eqref{E}--\eqref{IC} in $\UC$.

\smallskip
\noindent(b) There is at most one {\em distributional} solution $u$
of \eqref{E}--\eqref{IC} such that $u\in\EC$ and $\varphi(u)\in
L^2(Q_T)$.
\smallskip

\noindent(c) If in addition \eqref{holder} holds for some
  $\beta\in[\frac12,1]$, then there is at most one {\em distributional} solution
  $u$ of \eqref{E}--\eqref{IC} such that $u\in \EC$.
\end{corollary}
Note that we have uniqueness in a smaller class than we have existence for
by Theorem \ref{ex_Dsol}.
This uniqueness result should also be compared to our recent general
uniqueness result from \cite{DTEnJa15}. 
\begin{theorem}[Uniqueness 6, Theorem 2.8 in \cite{DTEnJa15}]\label{uniqD}
Assume \textup{($\textup{A}_\varphi$)}, \textup{($\textup{A}_{\mu}$)}, and
\textup{($\textup{A}_{u_0}$)}. Then there is at most one \textup{distributional} solution $u$
of \eqref{E}--\eqref{IC} satisfying  
$$u\in L^1(Q_T)\cap L^\infty(Q_T)\cap 
C([0,T];L_\textup{loc}^1(\R^N)).$$
\end{theorem}
In particular, any solution from Theorem \ref{ex_Dsol} is
unique. This result is more general than Corollary
\ref{uniqueness_d}, but the proof is also more complicated. When
Corollary \ref{uniqueness_d} applies, a greatly simplified uniqueness
argument is available (as we have seen).

 In view of the equivalence in Theorem \ref{equivalencedistrenergy},
we can also transport results in the other direction: from
distributional solutions to energy solutions. First we obtain a new
existence result as an immediate consequence of Theorem 
\ref{ex_Dsol}. 
\begin{corollary}[Existence 2]\label{ex_energy}
Assume \textup{($\textup{A}_\varphi$)}, \textup{($\textup{A}_{\mu}$)}, and
\textup{($\textup{A}_{u_0}$)}. Then there exists an \textup{energy} solution $u$
of \eqref{E}--\eqref{IC} satisfying  
$$u\in L^1(Q_T)\cap L^\infty(Q_T)\cap 
C([0,T];L_\textup{loc}^1(\R^N)).$$
\end{corollary}

In the case of {\em $x$-independent}
  operators, this existence 
  result is much more general than the result given in Theorem 1.1 in
  \cite{PaQuRo16}.  
Uniqueness results for energy solutions of \eqref{E}--\eqref{IC} are
given in Corollary \ref{uniqueness}. 
These results hold for a smaller class of functions
than the above existence results. However, a (more) general uniqueness
result can be obtained from the result for distributional
solutions in Theorem \ref{uniqD} and equivalence. 
\begin{corollary}[Uniqueness 6]
Assume \textup{($\textup{A}_\varphi$)}, \textup{($\textup{A}_{\mu}$)}, and
\textup{($\textup{A}_{u_0}$)}. Then there is at most one \textup{energy} solution $u$
of \eqref{E}--\eqref{IC} satisfying  
$$u\in L^1(Q_T)\cap L^\infty(Q_T)\cap 
C([0,T];L_\textup{loc}^1(\R^N)).$$
\end{corollary}
The proof is immediate. The solutions of Corollary \ref{ex_energy}
are therefore unique, and this 
result is stronger than the Ole\u{\i}nik type result Corollary \ref{uniqueness}.
In view of the well-posedness of both energy and
distributional solutions and the equivalence between the two notions
of solutions, we now have a full equivalence result under assumptions
\textup{($\textup{A}_\varphi$)}, \textup{($\textup{A}_{\mu}$)}, and
\textup{($\textup{A}_{u_0}$)}.

\begin{corollary}[Equivalent notions of solutions 2]
Assume \textup{($\textup{A}_\varphi$)}, \textup{($\textup{A}_{\mu}$)},
\textup{($\textup{A}_{u_0}$)}, and $u\in L^1(Q_T)\cap L^\infty(Q_T)\cap 
C([0,T];L_\textup{loc}^1(\R^N))$. Then $u$ is an {\em energy} solution
of \eqref{E}--\eqref{IC} if and only if  it is a {\em
  distributional} solution.
\end{corollary}

We end this section by new quantitative energy and related $L^p$-estimates for the unique distributional solution $u$ provided by Theorems
\ref{ex_Dsol} and \ref{uniqD}. This type of estimates are widely used
for different local and nonlocal equations of porous medium type, see
the discussion in Section \ref{sec:rem}. All proofs are given in Section \ref{sec:approxEconv}. Now, define $\Phi:\R\to\R$ by $\Phi(w)=\int_0^{w}\varphi(\xi)\dd
\xi$. Then we have:
\begin{theorem}[Energy inequality]\label{e-estim}
Assume \textup{($\textup{A}_\varphi$)}, \textup{($\textup{A}_{\mu}$)}, and
\textup{($\textup{A}_{u_0}$)}. Then the \textup{distributional} solution $u$ of \eqref{E}--\eqref{IC} satisfies
\begin{align*}
 \int_{\R^N}\Phi(u(x,\tau))\dd
 x+|\varphi(u)|_{\tau,E_\mu}^2\leq
 \int_{\R^N}\Phi(u_0)\dd x\qquad\text{for}\qquad \tau\in(0,T].
\end{align*}
\end{theorem}

 Since $\Phi\geq0$, we immediately have a quantitative bound on the energy.
\begin{corollary}
Assume \textup{($\textup{A}_\varphi$)}, \textup{($\textup{A}_{\mu}$)}, and
\textup{($\textup{A}_{u_0}$)}. Then the \textup{distributional} solution $u$ of \eqref{E}--\eqref{IC} satisfies
\begin{align*}
|\varphi(u)|_{T,E_\mu}^2\leq
\|\varphi(u_0)\|_{L^\infty(\R^N)}\|u_0\|_{L^1(\R^N)}<\infty. 
\end{align*}
\end{corollary}
There is also a second type of energy inequality that implies  $L^p$-bounds.

\begin{theorem}[$L^p$-bounds]\label{Lp-estim}
Assume \textup{($\textup{A}_\varphi$)}, \textup{($\textup{A}_{\mu}$)}, and
\textup{($\textup{A}_{u_0}$)}. Then the \textup{distributional} solution $u$ of \eqref{E}--\eqref{IC} satisfies, for $0<\tau\leq T$,
\begin{align*}
 \int_{\R^N}|u(x,\tau)|^p\dd
 x\leq
 \int_{\R^N}|u_0(x)|^p\dd x\qquad\text{for all}\qquad p\in[1,\infty), 
\end{align*}
and in the case $p=\infty$,
$$
\|u(\cdot,\tau)\|_{L^\infty(\R^N)}\leq\|u_0\|_{L^\infty(\R^N)}.
$$
\end{theorem}

\subsection{Remarks}
\label{sec:rem}
 
\subsubsection*{Locally shift-bounded kernels}
Let $\mu(\dd z)$ be a nonnegative locally finite Borel measure on
$\R^N\setminus\{0\}$ and $j(x,z)$ a measurable function satisfying
$$0<m\leq j(x,z)\leq M<\infty.$$
Then the kernel
 $$\lambda(x,\dd z)=j(x,z)\mu(\dd z)$$
is not only locally, but also globally, shift-bounded in
the sense that for all $x,h\in\R^N$ and Borel $B\subset
\R^N\setminus\{0\}$,  
$$
\frac{\lambda(x+h,B)}{\lambda(x,B)}\leq\frac{M}{m}.
$$
Examples of $\mu$ are L\'evy measures of L\'evy
processes, e.g. $\mu(\dd z)=\frac{c_{N,\alpha}}{|z|^{N+\alpha}}\dd z$ for
the $\alpha$-stable process ($\alpha\in(0,2)$) with the fractional
Laplacian as generator. The latter case
  corresponds exactly to assumption \textup{(A$_{\kernel 1}$'')}.

\subsubsection*{Recurrence and alternative characterization of $X$} 
In Theorem \ref{prop:characterizationofX} (a) approximation by test
functions is obtained by 
an additional assumption on the function
class. Alternatively, as in part (b), we can keep the
  original function class, but restrict the bilinear form 
$\mathcal{E}_\kernel$ (and hence the generator $A^\kernel$). In the
elliptic setting such results are given in  Theorem 3.2 in \cite{ScUe12}
under the assumptions that \textup{(A$_{\kernel 0}$)}, \textup{(A$_{\kernel1}$')} and \textup{(A$_{\lambda 2}$)} hold and the
closure of $(\mathcal{E}_\kernel, C_{\textup{c}}^\infty(\R^N))$  is {\em recurrent}.
A condition ensuring recurrence for symmetric L\'evy
processes is given in Section 37 in \cite{Sat99}. E.g. the fractional
Laplacian $-(-\Delta)^{\frac\alp 2}$ for $\alpha\in(0,2)$ is recurrent
if $N\leq \alpha$ -- which is a rather restrictive assumption! 
Similar results are true in our parabolic
setting. Assuming recurrence, or rather, assuming existence of the
  sequence of cut-off functions mentioned in Lemma 3.1 in  \cite{ScUe12}, we get 
$$
X=L^\infty(Q_T)\cap L^2(0,T;E_\kernel(\R^N)).
$$
The proof is an easy modification of the proof of Theorem 3.2 in
\cite{ScUe12} if we assume \text{\textup{(A$_{\lambda 0}$)}}, \textup{(A$_{\lambda 1}$')}, and \textup{(A$_{\lambda 2}$)} hold (as in Theorem \ref{prop:characterizationofX} (a)) and, in addition, $\iint |x-y|^2\Lambda(\dd x,\dd y)<\infty$. 
Note that the latter condition implies $\int_0^T\mathcal
E_\kernel[\phi_n,\phi_n]\dd t\to
0$ for any $\phi_n\in C^1_\textup{c}(Q_T)$ such that $\phi_n\to 1$ $a.e.$ and
$\|D\phi_n\|_{L^\infty}\to0$. However, this extra condition 
excludes all L\'evy processes and all $x$-independent generators.

\subsection*{Integral representations of the operators
  $A^\kernel$} 
In general the operator $A^\kernel$ is abstractly  defined from
$\mathcal{E}_\kernel$ by formula \eqref{defgen}. However explicit
integral representation formulas exist under additional assumptions on 
the kernel $\lambda(x,\dd z)$ (cf. \eqref{energy}). We follow
\cite{ScUe07} and assume (A$_{\kernel 0}$) and (A$_{\kernel 1}$)
hold and
$$
\kernel(x,\dd z)=\tilde{\lambda}(x,x+\dd z)=j(x,x+z)\dd z,
$$ 
where  $j\geq0$ is a symmetric measurable function on $\R^N\times\R^N\setminus\{x\}$ such that
$$
\int_{|z|\leq1}|z|\big|j(x,x+z)-j(x,x-z)\big|\dd z<\infty.
$$
Symmetric here means that $j(x,y)=j(y,x)$.
Note that now $\Lambda(\dd x,\dd y)=j(x,y)\, \dd x\dd y$ in \eqref{energy0}.
By Theorem 2.2 in \cite{ScUe07}, it then follows that
\begin{equation*}
\begin{split}
A^\kernel[\phi](x):=&\int_{|z|>0}\big(\phi(x+z)-\phi(x)-z\cdot D\phi(x)\mathbf{1}_{|z|\leq1}\big)\,j(x,x+z)\dd z\\
&\quad+\frac{1}{2}\int_{|z|>0}z\mathbf{1}_{|z|\leq1}\big(j(x,x+z)-j(x,x-z)\big)\dd
z\cdot D\phi(x)
\end{split}
\end{equation*}
for $\phi\in C_{\textup{c}}^2(\R^N)$. Compare with \eqref{deflevy} and note that
the second integral is like a drift term that vanishes if
$j(x,x+z)=j(x,x-z)$. Under slightly stronger assumptions, this
$A^\kernel$ coincides on $C_{\textup{c}}^2(\R^N)$ with the generator of
the closure of $(\mathcal{E}_\kernel, C_{\textup{c}}^\infty(\R^N))$ in
$L^2(\R^N)$ -- see Proposition 2.5 in \cite{ScUe07}.

Let us simplify and assume that
$$j(x,y)=j_1(x,y)\mu(x-y)$$
for $j_1$ symmetric, $j_1(x,x+z)=j_1(x,x-z)$, and $\mu$ even. This $j$
is symmetric and $j(x,x+z)=j(x,x-z)$.  Taking $j_1(x,y)=a(x)+a(y)$ and
$\mu(z)=\frac{c_{N,\alpha}}{|z|^{N+\alp}}$, the L\'evy density of the
fractional Laplacian, we get an $x$-depending
fractional Laplace like operator:
\begin{align*}A^\kernel_1[\phi](x)=&-a(x)(-\Delta)^{\frac\alp2}\phi(x) \\
&+ \int_{|z|>0}\Big(\phi(x+z)-\phi(x)-z\cdot
D\phi(x)\mathbf{1}_{|z|\leq1}\Big)
  a(x+z)\frac{c_{N,\alpha}}{|z|^{N+\alp}}\dd z.
\end{align*}
From this example we also learn that our class of operators does not include
the simplest and most natural $x$-depending fractional Laplace operator, 
$$-a(x)(-\Delta)^{\frac\alp2}\phi(x),$$
since it only satisfies the symmetry assumption on $j$ (or (A$_{\kernel 2}$)) 
 if $a$ is constant! 

\subsubsection*{On $L^p$-estimates}
If $\varphi(u)=u|u|^{m-1}$ and $\Levy^\mu=-(-\Delta)^{\frac{\sigma}{2}}$,
then by \cite{PaQuRoVa12} the estimate corresponding to Theorem
\ref{Lp-estim} takes the form 
\begin{equation}\label{Lpvaz}
\int_{\R^N}|u(x,\tau)|^p\dd x+\int_0^\tau \int_{\R^N} \left|(-\Delta)^{\frac{\sigma}{4}} |u|^{\frac{p+m-1}{2}}\right|^2\dd x\dd t \leq \int_{\R^N}|u_0(x)|^p\dd x.
\end{equation}
Note the additional energy term. A closer look
at our proof, see Corollary \ref{approxLpbound} and the proof of Theorem \ref{Lp-estim}, reveals that we could
also have an $L^p$-estimate with some energy. However, this energy is
only a limit and hard to characterize under our weak assumptions.

Such $L^p$ type decay estimates are an essential tool for nonlinear
diffusion equations of porous medium type. They imply that
$|u|^{\frac{p+m-1}{2}}$ belongs to some Sobolev space. This estimate and the
Nash-Gagliardo-Nirenberg inequality can be used in a Moser iteration
argument to obtain an $L^1-L^\infty$ smoothing effect and then
existence of energy solutions with initial data merely in $L^1$
\cite{Vaz06, Vaz07,PaQuRoVa12,PaQuRoVa14,PaQuRo16}. The
other main application of the $L^p$-energy estimates is as key steps
in Sobolev or Simon type compactness arguments. Such arguments are
used in \cite{Vaz06,BiKaMo10, BiImKa15,StanTesoVazCRAS, StTeVa15, StTeVa16}
to prove existence of energy solutions through the resolution of a
sequence of smooth approximate problems and passing to the limit in
view of compactness.  


\section{Proof of uniqueness for energy solutions}
\label{sec:proofofgeneraluniquenessx}

In this section we prove Theorem \ref{generaluniquenessx}.
We start by some preliminary results.

\begin{lemma}[Cauchy-Schwartz]\label{C-S}
Assume \textup{(A$_{\lambda 1}$)}. If $f,g\in L^2(0,T;E_\kernel(\R^N))$, then 
$$\left|\int_0^T\El[f(\cdot,t),g(\cdot,t)]\,\dd t\right|\leq |f|_{T,E_\kernel}|g|_{T,E_\kernel}.$$
\end{lemma}

The proof is as for the classical Cauchy-Schwartz and we omit it.

\begin{lemma}\label{T-lem}
Assume \textup{(A$_{\lambda 1}$)}. If $f\in L^2(0,T;E_\lambda(\R^N))$ and
$g(x,t)=\int_t^Tf(x,s)\dd s$, then $|g|_{T,E_\kernel}^2\leq
\frac{T^2}2|f|_{T,E_\kernel}^2$. 
\end{lemma}

\begin{proof}
By Jensen's inequality and Tonelli's lemma, 
\begin{align*}
|g|_{T,E_\kernel}^2
&\leq\int_0^T\frac{1}{2}\iint(T-t)\int_t^T\left|f(x+z,s)-f(x,s)\right|^2\, \dd s\kernel(x, \dd z)\dd x\dd t\\
&=\int_0^T(T-t)\Big(\int_t^T|f(\cdot,s)|_{E_\lambda}^2\dd s\Big) \dd t,
\end{align*}
and the result follows.
\end{proof}

Since an energy solution has some regularity, the weak formulation of
the equation will hold also with less regular test functions. We will now
formulate such a type of result in the relevant setting for the
Ole\u{\i}nik  argument.
\begin{lemma}\label{density} Let $u$ be an energy solution of \eqref{EE}--\eqref{EIC}. If $u\in L^1(Q_T)$, $u_0\in L^1(\R^N)$, and
  $\varphi(u)\in L^2(0,T;E_\lambda(\R^N))$, then for any $\phi\in X$,
\begin{align*}
\displaystyle\int_{0}^T\left(-\int_{\R^N}u\phi\,\dd
   x-\mathcal{E}_\kernel\left[\varphi(u),\int_t^T\phi(\cdot,s)\dd s\right]
 \right)\dd t&\\
+
  \int_{\R^N}u_0(x)\left(\int_0^T\phi(x,s)\,\dd s\right)\dd x&=0.
\end{align*}
\end{lemma}
In other words, we may take $\psi(x,t)=\int_t^T\phi(x,s)\,\dd s$ in
Definition \ref{energysoln} for $\phi\in X$. Note that the integrals are
well-defined: see Lemma \ref{T-lem}. From the proof below it follows
that the choice of space $X$ is (close to) optimal.

\begin{proof}
From the definition of $X$ there is $C_{\textup{c}}^\infty(\R^N\times
[0,T))\ni\phi_n\to \phi\in X$ for the convergence in $X$ as
$n\to\infty$. Let 
$$\psi(x,t):=\int_t^T\phi(x,s)\,\dd s\qquad\text{and}\qquad\psi_n(x,t):= \int_t^T\phi_n(x,s)\,\dd s.
$$
Observe that $\psi_n\in C_{\textup{c}}^\infty(\R^N\times [0,T))$ since $\phi_n$ is. 
By Cauchy-Schwartz' inequality, Lemma \ref{T-lem}, and the convergence in
$X$, we see that
\begin{align*}
\int_0^T\mathcal{E}_\kernel\Big[\varphi(u),(\psi_n-\psi)\Big]
\dd t \leq |\varphi(u)|_{T,E_\lambda}\frac{T^2}2|\phi_n-\phi|_{T,E_\lambda} \to
0\quad\text{as}\quad n\to\infty.
\end{align*}
Since $u\in L^1(Q_T)$ and $\phi_n$ converges in $X$, we also have
\begin{align*}\iint_{Q_T}u(\dell_t\psi_n-\dell_t\psi)\,\dd x\dd t
&=-\iint_{Q_T}u(\phi_n-\phi)\,\dd x\dd t\to  0\qquad\text{as}\qquad n\to\infty.
\end{align*}
In a similar way, $\int_{\R^N}u_0(x)(\psi_n-\psi)(x,0) \,\dd x\to 0$. 
The result now follows from taking $\psi=\psi_n$ in the definition of
energy solutions (Definition \ref{energysoln}), and using the above
estimates to pass to the limit.
\end{proof}

\begin{remark}
A closer inspection of the proof reveals that strong
$|\cdot|_{T,E_\lambda}$ convergence cannot be replaced
by the corresponding weak convergence. The reason is that the 
weak convergence property for the test functions $\phi_n$ is lost when they
are integrated in time to yield the $\psi_n$'s. 
\end{remark}

Note that for the proof of Lemma \ref{density}, the definition of $X$ is {\em essential} in the sense that we take those functions which can be approximated by $C_{\textup{c}}^\infty$-functions.
This lemma is crucial in the Ole\u{\i}nik  argument below because we want to take 
$$\psi(x,t)=\int_t^T(\varphi(u)-\varphi(v))(x,s)\,\dd s$$
as a test function. By Lemma \ref{density}, we need that
$\varphi(u),\varphi(v)\in X$ for this to be possible, and  
this explains this strange assumption and space.

\begin{proof}[Proof of Theorem \ref{generaluniquenessx} (Uniqueness 1)]
Assume there are two different energy solutions $u$ and $v$ of
\eqref{EE} with the same initial data \eqref{EIC}. Let
$U=u-v$ and $\Phi=\varphi(u)-\varphi(v)$, and note that the proof is
complete if we can show that $U=0$ a.e. in $Q_T$. 

To show that, we subtract the energy formulation of the equations for
$u$ and $v$ (Definition 
\ref{energysoln}). Since the initial data are the same, this leads to
\begin{equation}\label{energyformulation}
\int_{0}^T\left(\int_{\R^N}U\dell_t\psi\dd
x-\mathcal{E}_\kernel[\Phi,\psi] \right)\dd t=0 \quad\text{for
  all}\quad\psi\in C_{\textup{c}}^\infty(\R^ N\times[0,T)).
\end{equation}
Now we adapt the classical argument of Ole\u{\i}nik et al. \cite{OlKaCz58}
and seek to take
\begin{equation*}
\zeta(x,t)=\begin{cases}
\int_t^T\Phi(x,s)\,\dd s & 0\leq t< T\\
0 & t\geq T,
\end{cases}
\end{equation*}
as a test function. Since $\Phi\in X$ (by the definition of $\UC$), this can be done by Lemma
\ref{density}, and hence
\begin{equation}\label{energyformulation2}
\int_{0}^T\left(-\int_{\R^N}U\Phi\,\dd
x-\mathcal{E}_\kernel[\Phi,\zeta] \right)\dd t=0.
\end{equation}

Since $\int_0^T|\mathcal{E}_\kernel[\Phi,\zeta]| \dd t<\infty$ by
Lemma \ref{T-lem}, we have by Fubini's theorem
\begin{equation*}
\begin{split}
&\int_0^T\mathcal{E}_\kernel[\Phi,\zeta] \dd t\\
&=\frac{1}{2}\int_0^T\int_{\R^N}\int_{|z|>0}\big(\Phi(x+z,t)-\Phi(x,t)\big)\big(\zeta(x+z)-\zeta(x)\big)\,\kernel(x, \dd z)\dd x\dd t\\
&=\frac{1}{2}\int_{\R^N}\int_{|z|>0}\int_0^T\big(\Phi(x+z,t)-\Phi(x,t)\big)\times\\
  &\qquad\qquad\qquad\qquad\times\int_t^T\big(\Phi(x+z,s)-\Phi(x,s)\big)\, \dd s\dd t\kernel(x, \dd z)\dd x. \\
\end{split}
\end{equation*}
Then by the identity for $F\in L^1((0,T))$,
\begin{equation*}
\int_0^TF(t)\left(\int_t^TF(s)\dd s\right)\dd t=\int_0^T\int_t^TF(t)F(s)\,\dd s\dd t=\frac{1}{2}\left(\int_0^TF(t)\,\dd t\right)^2
\end{equation*}
(follows easily since $\int_0^T\int_t^T\dots
\dd s\dd t=\int_0^T\int_0^s\dots \dd t \dd s$), 
\begin{equation*}
\begin{split}
&\int_0^T\mathcal{E}_\kernel[\Phi,\zeta] \dd t=\frac{1}{4}\int_{\R^N}\int_{|z|>0}\left(\int_0^T\big(\Phi(x+z,t)-\Phi(x,t)\big)\dd t\right)^2\kernel(x, \dd z)\dd x\geq0.
\end{split}
\end{equation*}
Returning to \eqref{energyformulation2}, we then find that
\begin{equation*}
\int_{0}^T\int_{\R^N}U\Phi\,\dd x\dd t\leq0.
\end{equation*}

Since $\varphi$ is nondecreasing by \textup{($\textup{A}_\varphi$)},
$U\Phi\geq0$ a.e., and it then follows that $U\Phi=0$ a.e. in $Q_T$.
This means that at a.e. point, either $U=0$ or
$\Phi=0$, and hence since $U=0$ implies $\Phi=0$ by definition,
$$\Phi=0\qquad a.e.\quad\text{in}\quad Q_T.$$
 Then by equation \eqref{energyformulation},
$$
\int_0^T\int_{\R^N}U\dell_t\psi\,\dd x\dd t=0 \qquad\text{for
  all}\qquad \psi\in C_{\textup{c}}^\infty(\R^N\times[0,T)).$$
Since $\psi(x,t):=\int_t^T\phi(x,s)\,\dd s\in C_{\textup{c}}^\infty(\R^N\times[0,T))$ for arbitrary $\phi\in
C_{\textup{c}}^\infty(Q_T)$,
$$
-\int_0^T\int_{\R^N}U\phi\,\dd x\dd t=0 \qquad\text{for
  all}\qquad \phi\in C_{\textup{c}}^\infty(Q_T),$$
and hence $U=0$ a.e. in $Q_T$ by du Bois-Reymond's lemma.
\end{proof}


\section{Distributional solutions with finite energy}
\label{sec:propdistsoln}

Our main focus in this section is to prove Theorems \ref{equivalencedistrenergy}, \ref{ex_Dsol}, \ref{e-estim}, and \ref{Lp-estim}. First, we prove the equivalence of notions of solutions. Second, we consider an approximate problem of \eqref{E}--\eqref{IC}. The energy and $L^p$-estimates are then shown to hold for the solution of that problem. A compactness result will give us convergence of solutions of the approximate problem, and we thus obtain existence of some limit solution of the full problem satisfying the same estimates.

We recall that (i) $\Levy^\mu[\psi]$ is well-defined for $\psi\in C^2(\R^N)\cap L^\infty(\R^N)$; (ii) $\Levy^\mu[\psi]$ is bounded in $L^1/L^\infty$ for $\psi\in W^{2,1}/W^{2,\infty}$; and (iii) $\Levy^\mu$ is symmetric for e.g. functions in $W^{2,1}/W^{2,\infty}$ (see Lemma 3.5 in \cite{DTEnJa15}). Note also that for $\mu$ replaced by $\mu_r:=\mu\mathbf{1}_{|z|>r}$, (i)--(iii) holds when we only assume that $\psi$ is in $L^\infty$, $L^1/L^\infty$, and $L^1/L^\infty$ (see Remark 3.6 (b) in \cite{DTEnJa15}). 

\subsection{Equivalent notions of solutions} 
\label{sec:equivnotionofsolutions}

We establish the relation between the ($x$-independent) bilinear form and our L\'evy operator, as a consequence, we get equivalence of energy and distributional solutions under certain conditions.

\begin{proposition}\label{bilinearenergy}
Assume \textup{(\textup{A}$_\mu$)}. For any $\psi\in C_{\textup{c}}^\infty(\R^N)$, and $v\in L^\infty(\R^N)\cap E_\mu(\R^N)$, we have
\begin{equation*}
\begin{split}
&\int_{\R^N}v(x)\Levy^{\mu}[\psi](x)\,\dd x\\
&=-\frac{1}{2}\int_{\R^N}\int_{|z|>0}\big(v(x+z)-v(x)\big)\big(\psi(x+z)-\psi(x)\big)\,\mu(\dd z)\dd x=-\mathcal{E}_\mu[v,\psi].
\end{split}
\end{equation*}
\end{proposition}

\begin{remark}
The result holds as long as both sides make sense.
\end{remark}

\begin{lemma}\label{simpleStroock-Var}
Assume that $\nu\geq0$ is an even Radon measure with $\nu(\R^N)<\infty$, and $1\leq p,q\leq \infty$ with $\frac{1}{p}+\frac{1}{q}=1$. For any $f\in L^p(\R^N)$ and $g\in L^q(\R^N)$, we have
$$
\int_{\R^N}g(x)\Levy^{\nu}[f](x)\,\dd x=-\mathcal{E}_\nu[f,g].
$$
\end{lemma}

This proof is postponed to Appendix \ref{sec:technicalresult}. 

\begin{proof}[Proof of Proposition \ref{bilinearenergy}]
Replace $\nu$ by $\mu_r=\mu\mathbf{1}_{|z|>r}$ in Lemma \ref{simpleStroock-Var}, and let $g=v$ and $f=\psi$. Then the result follows by Lebesgue's dominated convergence theorem as $r\to0^+$ since $\mathbf{1}_{|z|>r}\leq 1$.
\end{proof}

\begin{proof}[Proof of Theorem \ref{equivalencedistrenergy} (Equivalent notions of solutions)]
\noindent $(a)\implies (b)$ In Definition \ref{energysoln}, we have that $|\varphi(u)|_{T,E_{\mu}}<\infty$, and then we can use Proposition \ref{bilinearenergy} to obtain (note that $\varphi(u)\in L^\infty(\R^N)$)
\begin{equation*}
\int_0^T\int_{\R^N}u\dell_t\psi +\varphi(u)\Levy^\mu[\psi]\dd x\dd t+\int_{\R^N}u_0(x)\psi(x,0)\dd x=0 \text{ }\forall\psi\in
  C_\textup{c}^\infty(\R^N\times[0,T)).
\end{equation*}

\medskip
\noindent $(b) \implies (a)$ We write Definition \ref{distsol} in the following way
\begin{equation*}
\begin{split}
\int_0^T\bigg(\int_{\R^N}u\dell_t\psi\dd x &+ \int_{\R^N}\varphi(u)\Levy^{\mu}[\psi] \dd x\bigg)\dd t\\
&\qquad+\int_{\R^N}u_0(x)\psi(x,0)\dd x=0 \quad\forall\psi\in C_\textup{c}^\infty(\R^N\times[0,T)).
\end{split}
\end{equation*}
By the assumptions, $|\varphi(u)|_{T,E_\mu}<\infty$, and hence, we can
use Proposition \ref{bilinearenergy} in the other direction to get
energy solutions. 
\end{proof}

\subsection{The approximate problem of \eqref{E}--\eqref{IC}}
\label{sec:approxEconv}

By using a priori and existence results for a simplified version of \eqref{E}--\eqref{IC}, we can take the limit of a sequence of solutions of such problems, and then conclude that some limit solution of the full problem exists and enjoys the energy and $L^p$-estimates.

Let $\omega_n$ be a family of mollifiers defined by
\begin{align}
\omega_n(\sigma):=n^N\omega\left(n\sigma\right)
\label{mollifierspace}
\end{align}
for fixed $0\leq\omega\in C_\textup{c}^\infty(\R^N)$ with $\text{supp}\,
\omega\subseteq \overline{B}(0,1)$, $\omega(\sigma)=\omega(-\sigma),$ 
$\int\omega=1$, and define
\begin{equation}\label{phiapprox}
\varphi_n(x):=\varphi\ast\omega_n(x)-\varphi\ast\omega_n(0)\text{ where $\omega_n$ is given by \eqref{mollifierspace} with $N=1$.}
\end{equation}
Now, consider the following approximation of \eqref{E}--\eqref{IC} where the measure $\mu$ is replaced by $\mu_r=\mu\mathbf{1}_{|z|>r}$ and the nonlinear diffusion flux $\varphi$ is replaced by $\varphi_n$:
\begin{align}\label{approxE}
\dell_t u_{r,n}-\Levy^{\mu_r}[\varphi_{n}(u_{r,n})]&=0 && \text{in}\quad Q_T,\\
u_{r,n}(x,0)&=u_0(x) &&\text{on}\quad \R^N, \label{approxIC}
\end{align}
with
\begin{equation*}
\Levy^{\mu_r}[\psi](x)=\int_{|z|>0}\big(\psi(x+z)-\psi(x)\big) \,\mu_r(\dd z).
\end{equation*}
Note that $\varphi_n\in C^\infty(\R)$ (and hence, locally Lipschitz), $\varphi_n(0)=0$, and $\varphi_n\to\varphi$ locally uniformly on $\R$ by \textup{(\textup{A}$_\varphi$)}, the properties of mollifiers, and Remark \ref{consequencesassumptions} (f). Furthermore, recall that for any $r>0$, the operator $\Levy^{\mu_r}[\psi]$ is well-defined for merely bounded $\psi$.

\begin{remark}\label{rem:propapproxE}
Since \eqref{approxE}--\eqref{approxIC} is just a special case of \eqref{E}--\eqref{IC}, existence, uniqueness, (uniform) $L^1$-, $L^\infty$-bounds, and time regularity holds for \eqref{approxE}--\eqref{approxIC} by Theorem 2.10 in \cite{DTEnJa15} or by \cite{DTEnJa16} through limit procedures and compactness results for entropy or numerical solutions.
\end{remark}

\begin{theorem}[Existence and uniqueness, Theorem 2.8 in \cite{DTEnJa15}]\label{ex:approx}
Assume \textup{(\textup{A}$_\varphi$)}, \textup{(\textup{A}$_\mu$)}, and
\textup{(\textup{A}$_{u_0}$)}. Then there exists a unique distributional solution $u_{r,n}$
of \eqref{approxE}--\eqref{approxIC} satisfying  
$$u_{r,n}\in L^1(Q_T)\cap L^\infty(Q_T)\cap 
C([0,T];L_\textup{loc}^1(\R^N)).$$
\end{theorem}

Now we first prove that \eqref{approxE} holds a.e., and then we deduce energy and clean $L^p$-estimates (the latter by a Stroock-Varopoulos type result) from the rather general inequality in Proposition \ref{Psiineq} below. 

\begin{lemma}\label{timeregularityapproxE}
Assume \textup{(\textup{A}$_\varphi$)}, \textup{(\textup{A}$_\mu$)}, and \textup{(\textup{A}$_{u_0}$)}. Then the distributional solution $u_{r,n}$ of \eqref{approxE}--\eqref{approxIC}
with initial data $u_0$ satisfies
$$
\dell_tu_{r,n}\in L^1(Q_T)\cap L^\infty(Q_T)\qquad\text{and}\qquad\dell_t u_{r,n}=\Levy^{\mu_r}[\varphi_n(u_{r,n})]\text{ a.e. in }Q_T.
$$
\end{lemma}

\begin{proof}
By the definition of distributional solutions for \eqref{approxE}--\eqref{approxIC} and the symmetry of $\Levy^{\mu_r}$,
$$
-\iint_{Q_T}u_{r,n}\dell_t\psi\, \dd x\dd t=\iint_{Q_T}\varphi_n(u_{r,n})\Levy^{\mu_r}[\psi]\, \dd x \dd t=\iint_{Q_T}\Levy^{\mu_r}[\varphi_n(u_{r,n})]\psi\, \dd x \dd t.
$$
Hence, $\Levy^{\mu_r}[\varphi_n(u_{r,n})]$ is the weak time derivative of $u_{r,n}$. Since $\varphi_n\in W^{1,\infty}(\R)$, $\varphi_n(u_{r,n})\in L^1\cap L^\infty$ and hence, we get that
$g:=\Levy^{\mu_r}[\varphi_n(u_{r,n})]\in L^1(Q_T)\cap L^\infty(Q_T)$.

Assume also $u_{r,n}\in C^1$. Then $\partial_t u_{r,n}=g$ and we can use
the Fundamental theorem of calculus to see that
$$
\left\|\frac{u_{r,n}(\cdot,\cdot+h)-u_{r,n}}{h}-g\right\|_{L^1(Q_T)}\leq\int_0^1\|g(\cdot,\cdot+sh)-g\|_{L^1(Q_T)}\dd s.
$$
By an approximation argument in $L^1$, this inequality holds also without
the $C^1$ assumption.
Taking the limit as $h\to0^+$ (the right-hand side goes to zero by Lebesgue's dominated convergence theorem since translations in $L^1$ are continuous), we obtain that
$$
\lim_{h\to0^+}\frac{u_{r,n}(x,t+h)-u_{r,n}(x,t)}{h}=g(x,t)\qquad\text{in $L^1(Q_T)$},
$$
and hence, $\dell_tu_{r,n}$ exists and equals $g$ a.e. in $Q_T$.
\end{proof}

To prove the next result, we need to define cut-off functions: Consider $\mathcal{X}\in C_{\textup{c}}^\infty(\R^N)$ such that $\mathcal{X}\geq0$, $\mathcal{X}=1$ when $|x|\leq1$, and $\mathcal{X}=0$ when $|x|>2$,
and define 
\begin{equation}\label{standardcut-off}
\mathcal{X}_R(\cdot):=\mathcal{X}\left(\frac{\cdot}{R}\right)\in C_{\textup{c}}^\infty(\R^N)\qquad\text{for}\qquad R>0. 
\end{equation}

\begin{proposition}\label{Psiineq}
Assume \textup{(\textup{A}$_\varphi$)}, \textup{(\textup{A}$_\mu$)}, \textup{(\textup{A}$_{u_0}$)}, and $0<\tau\leq T$. Let $\Psi\in W_\textup{loc}^{1,\infty}(\R)$ with $\Psi(0)=0$. Then the distributional solution $u_{r,n}$
of \eqref{approxE}--\eqref{approxIC} satisfies
\begin{equation*}
\begin{split}
&\int_{\R^N}\Psi(u_{r,n}(x,\tau))\dd x-\int_0^\tau\int_{\R^N}\Psi'(u_{r,n}(x,t))\Levy^{\mu_r}[\varphi_n(u_{r,n}(x,t))]\dd x\dd t\\
&\quad =\int_{\R^N}\Psi(u_0(x))\dd x.
\end{split}
\end{equation*}
\end{proposition}

\begin{remark}
On page 1256 in \cite{PaQuRoVa12}, a similar result as the above is obtained for $\Psi(u)$ nonnegative, nondecreasing and convex.
\end{remark}

\begin{proof}
Observe that we may assume $\Psi\in W^{1,\infty}(\R)$ since $\Psi\in W_\textup{loc}^{1,\infty}(\R)$ and $u_{r,n},u_0\in L^\infty$. By Lemma \ref{timeregularityapproxE}, $\dell_tu_{r,n}=\Levy^{\mu_r}[\varphi_n(u_{r,n})]$ a.e. in $Q_T$. Multiply this a.e.-equation by $\Psi'(u_{r,n}(x,t))\mathcal{X}_R(x)$ (where $\mathcal{X}_R$ is defined in \eqref{standardcut-off}) and integrate (in $x$) over $\R^N$ to get
\begin{equation*}
\int_{\R^N}\dell_tu_{r,n}\Psi'(u_{r,n})\mathcal{X}_R\,\dd x=\int_{\R^N}\Levy^{\mu_r}[\varphi_n(u_{r,n})]\Psi'(u_{r,n})\mathcal{X}_R\,\dd x.
\end{equation*}
By Lemma \ref{timeregularityapproxE} and the Sobolev chain rule given by Theorem 2.1.11 in \cite{Zie89}, the left-hand side equals $\int_{\R^N}\dell_t\Psi(u_{r,n})\mathcal{X}_R\,\dd x$. 
Note that the function $\mathcal{X}_R$ converges pointwise to $1$, is bounded by $1$, and is integrable.
Hence we can move the time derivative outside the integral on the left-hand side by Lebesgue's dominated convergence theorem and Lemma \ref{timeregularityapproxE} since $|\dell_t\Psi(u_{r,n})\mathcal{X}_R|\in L^1(\R^N)$:
\begin{equation*}
\begin{split}
\frac{\dd}{\dd t}\int_{\R^N}\Psi(u_{r,n})\mathcal{X}_R\,\dd x=\int_{\R^N}\Levy^{\mu_r}[\varphi_n(u_{r,n})]\Psi'(u_{r,n})\mathcal{X}_R\,\dd x.
\end{split}
\end{equation*}

We integrate in time from $t=0$ to $t=\tau$, and use that
$u_{r,n}\in\linebreak C([0,T];L_\textup{loc}^1(\R^N))$ (cf. Theorem \ref{ex:approx}) and $\mathcal{X}_R\in C_{\textup{c}}^\infty(\R^N)$ to obtain
\begin{equation}\label{integratedintimeenergycutoff}
\begin{split}
&\int_{\R^N}\Psi(u_{r,n}(x,\tau))\mathcal{X}_R(x)\dd x-\int_{\R^N}\Psi(u_0(x))\mathcal{X}_R(x)\dd x\\
&=\int_0^\tau\int_{\R^N}\Levy^{\mu_r}[\varphi_n(u_{r,n}(\cdot,t))](x)\Psi'(u_{r,n}(x,t))\mathcal{X}_R(x)\dd x\dd t.
\end{split}
\end{equation}
Since $\Psi\in W^{1,\infty}(\R)$ and $\Psi(0)=0$,
$|\Psi(w)|\leq\|\Psi'(w)\|_{L^\infty}|w|\in L^1$ for
$w=u_{r,n},u_0$. Moreover, since $u_{r,n}, \varphi_n(u_{r,n})$ and hence also $\Levy^{\mu_r}[\varphi_n(u_{r,n})]$ is
integrable, we get
$|\Levy^{\mu_r}[\varphi_n(u_{r,n})]\Psi'(u_{r,n})\mathcal{X}_R|\in
L^1(\R^N\times(0,\tau))$.
Then Lebesgue's dominated convergence theorem can be used on both sides of \eqref{integratedintimeenergycutoff} as $R\to\infty$ to 
complete the proof.
\end{proof}

\begin{corollary}[Energy estimate]\label{energyapproxproblem}
Let $\Phi_n(w):=\int_0^{w}\varphi_{n}(\xi)\dd \xi$. Under the assumptions of Proposition \ref{Psiineq},
$$
\int_{\R^N}\Phi_n(u_{r,n}(x,\tau))\dd x+|\varphi_n(u_{r,n})|_{\tau,{E}_{\mu_r}}^2= \int_{\R^N}\Phi_n(u_{0}(x))\dd x.
$$
In particular, 
$$
|\varphi_n(u_{r,n})|_{\tau,{E}_{\mu_r}}\leq\int_{\R^N}\Phi_n(u_{0}(x))\dd x\leq\|\varphi_n(u_0)\|_{L^\infty(\R^N)}\|u_0\|_{L^1(\R^N)}<\infty.
$$
\end{corollary}

\begin{proof}
We observe that $\Phi_n:\R\to\R$ is $C^1$ and $\Phi_n(0)=0$. Moreover, $\Phi_n'(w)=\varphi_n(w)$ which is bounded when $w=:u_{r,n},u_0\in L^\infty$ by \textup{(\textup{A}$_\varphi$)} and \eqref{phiapprox}. Hence, $\Phi_n$ is Lipschitz, and thus, we can replace $\Psi$ by $\Phi_n$ in Proposition \ref{Psiineq} to get
\begin{equation*}
\begin{split}
&\int_{\R^N}\Phi_n(u_{r,n}(x,\tau))\dd x-\int_{0}^\tau\int_{\R^N}\varphi_n(u_{r,n}(x,t))\Levy^{\mu_r}[\varphi_n(u_{r,n}(\cdot,t))](x)\dd x\dd t\\
&=\int_{\R^N}\Phi_n(u_0(x))\dd x.
\end{split}
\end{equation*}

Since \eqref{phiapprox} hold and $\Levy^{\mu_r}[\varphi(u_{r,n})]$ is integrable, we conclude the first part by  Lemma \ref{simpleStroock-Var} (take $f=\varphi_n(u_{r,n})=g$).
For the last part, we use that $\Phi_n(u_0)=|\Phi_n(u_0)|\leq\|\Phi_n'(u_0)\|_{L^\infty}|u_0|$, and hence, since $\Phi_n\geq0$,
$$|\varphi_n(u_{r,n})|_{\tau,E_{\mu_r}}^2\leq\int_{\R^N}\Phi_n(u_0(x))\dd x\leq \|\varphi_n(u_0)\|_{L^\infty}\|u_0\|_{L^1}$$
which completes the proof.
\end{proof}

\begin{lemma}[General Stroock-Varopoulos]\label{genS-V}
Assume \textup{(A$_{\kernel1}$)}, $Q,R,S\in C^1(\R)$, $(S')^2\leq Q'R'$, and $|Q(\psi)|_{T,E_\kernel},|R(\psi)|_{T,E_\kernel}<\infty$ for some $\psi:Q_T\to\R$. Then
$$
\int_0^T\mathcal{E}_{\kernel}[Q(\psi(\cdot,t)),R(\psi(\cdot,t))]\dd t\geq |S(\psi)|_{T,E_\kernel}.
$$
\end{lemma}

\begin{proof}
Assume without loss of generality that $b>a$. By the Fundamental theorem of calculus, Cauchy-Schwartz' inequality, and $Q'R'\geq (S')^2$, we obtain
\begin{equation*}
\begin{split}
&(Q(b)-Q(a))(R(b)-R(a))=\int_a^b\left(\sqrt{Q'(t)}\right)^2\dd t\int_a^b\left(\sqrt{R'(t)}\right)^2\dd t\\
&\geq\left(\int_a^b\sqrt{Q'(t)R'(t)}\,\dd t\right)^2\geq\left(\int_a^bS'(t)\dd t\right)^2=(S(b)-S(a))^2.
\end{split}
\end{equation*}
By the definition of $\mathcal{E}_\kernel$ and $|\cdot|_{T,E_\kernel}$, the result follows.
\end{proof}

\begin{remark}\label{S-Vforfinitemes}
\begin{enumerate}[(a)]
\item See Proposition 4.11 in \cite{BrPa15} for a similar result.
\item Observe that the same lemma holds for a nonnegative even Radon measure $\nu$ with $\nu(\R^N)<\infty$ under the simplified assumption $Q(\psi)\in L^p(Q_T)$ and $R(\psi)\in L^q(Q_T)$ with $1\leq p,q\leq \infty$ and $\frac{1}{p}+\frac{1}{q}=1$.
\end{enumerate}
\end{remark}

\begin{corollary}[$L^p$-bound]\label{approxLpbound}
  Let $\Lambda(\xi)=|\xi|^p$ and
  $\Xi_n(w)=\int_0^{w}\sqrt{\Lambda''(\xi)\varphi_n'(\xi)}\dd \xi.$
  Under the assumptions of Proposition \ref{Psiineq} and
  $p\in(1,\infty)$,
\begin{equation*}
\begin{split}
\int_{\R^N}|u_{r,n}(x,\tau)|^p\dd x+|\Xi_n(u_{r,n})|_{\tau, E_{\mu_r}}^2\dd t\leq\int_{\R^N}|u_{0}(x)|^p\dd x
\end{split}
\end{equation*}
In particular,
$$
\int_{\R^N}|u_{r,n}(x,\tau)|^p\dd x\leq \int_{\R^N}|u_{0}(x)|^p\dd x<\infty.
$$
\end{corollary}

\begin{remark}
The above result also ensures that $|\Xi_n(u_{r,n})|_{\tau, E_{\mu_r}}^2$ is uniformly bounded in $r$ and $n$.
\end{remark}

\begin{proof}
Observe that $u_{r,n}\in L^p(Q_T)$ for $p\in(1,\infty)$ by standard interpolation in $L^p$-spaces. 

{\sc Case 1: $p\in[2,\infty)$.} The function $\Lambda$ is convex, $\Lambda\in W_\textup{loc}^{1,\infty}(\R)$, and $\Lambda(0)=0$. That is, we can replace $\Psi$ by $\Lambda$ in Proposition \ref{Psiineq} to get
\begin{equation}\label{intermediateLpbound}
\begin{split}
&\int_{\R^N}\Lambda(u_{r,n}(x,\tau))\dd x-\int_{0}^\tau\int_{\R^N}\Lambda'(u_{r,n}(x,t))\Levy^{\mu_r}[\varphi_n(u_{r,n}(\cdot,t))](x)\dd x\dd t\\
&=\int_{\R^N}\Lambda(u_0(x))\dd x
\end{split}
\end{equation}
Note that $\Lambda'(\xi)=p|\xi|^{p-2}\xi$ and $\Lambda''(\xi)=p(p-1)|\xi|^{p-2}$. Since $\Lambda'\in W_\textup{loc}^{1,\infty}(\R)$, $u_{r,n}\in L^2(Q_T)\cap L^\infty(Q_T)$, and \eqref{phiapprox} holds, $g:=\Lambda'(u_{r,n}(\cdot,t))\in L^2(\R^N)$ and\linebreak $f:=\varphi_n(u_{r,n}(\cdot,t))\in L^2(\R^N)$. By Lemma \ref{simpleStroock-Var}, 
\begin{equation*}
\begin{split}
-\int_{0}^\tau\int_{\R^N}\Lambda'(u_{r,n})\Levy^{\mu_r}[\varphi_n(u_{r,n})]\,\dd x\dd t=\int_0^\tau\mathcal{E}_{\mu_r}[\Lambda'(u_{r,n}),\varphi_n(u_{r,n})]\,\dd t.
\end{split}
\end{equation*}
Then by Lemma \ref{genS-V} and Remark \ref{S-Vforfinitemes} (b)
(take $Q:=\Lambda'$ and $R:=\varphi_n$), 
\begin{equation*}
\begin{split}
\int_0^\tau\mathcal{E}_{\mu_r}[\Lambda'(u_{r,n}(\cdot,t)),\varphi_n(u_{r,n}(\cdot,t))]\dd t\geq|\Xi_n(u_{r,n})|_{\tau, E_{\mu_r}}^2\geq0,
\end{split}
\end{equation*}
since $\Xi_n$ satisfies $(\Xi_n')^2\leq \Lambda''\varphi_n'$. Hence the corollary follows by \eqref{intermediateLpbound}. 

{\sc Case 2: $p\in(1, 2)$.} We follow the idea of the proof of Corollary 5.12 in \cite{BiImKa15}. For each $\delta>0$, consider the function $\Lambda_\delta$ such that
$$
\Lambda_\delta(0)=\Lambda'_\delta(0)=0\qquad\text{and}\qquad\Lambda''_\delta(\xi)=p(p-1)\left((\delta^2+\xi^2)^{\frac{p-2}{2}}-\delta^{p-2}\right).
$$
Note that $0\leq\Lambda''_\delta(\xi)\leq p(p-1)|\xi|^{p-2}$, and then, 
$$
|\Lambda'_\delta(\xi)|=\left|\int_0^\xi\Lambda''_\delta(s)\dd s\right|\leq p|\xi|^{p-1}\quad\text{and}\quad|\Lambda_\delta(\xi)|\leq\left|\int_0^\xi\Lambda'_\delta(s)\dd s\right|\leq |\xi|^{p}.
$$ 
Since $g:=\Lambda'_\delta(u_{r,n}(\cdot,t))\in L^\infty(\R^N)$ and $f:=\varphi_n(u_{r,n}(\cdot,t))\in L^1(\R^N)$, we get -- by following the calculations in Case 1 -- that 
\begin{equation}\label{p12normequality}
\int_{\R^N}\Lambda_\delta(u_{r,n}(x,\tau))\dd x+|\Xi_{n,\delta}(u_{r,n})|_{\tau,E_{\mu_r}}^2\leq \int_{\R^N}\Lambda_\delta(u_{0}(x))\dd x
\end{equation}
with
$$
\Xi_{n, \delta}(u_{r,n})=\int_0^{u_{r,n}}\sqrt{\Lambda''_\delta(\xi)\varphi_n'(\xi)}\dd \xi\geq0.
$$
By a direct argument, using $\Lambda''_\delta,\Lambda'',\varphi_n'\geq0$ and Cauchy-Schwartz's inequality, we obtain
\begin{equation}\label{pnormenergyfunction}
\begin{split}
\left|\Xi_{n, \delta}(u_{r,n})-\Xi_n(u_{r,n})\right|
&\leq\int_0^{u_{r,n}}\sqrt{|\Lambda''_\delta(\xi)-\Lambda''(\xi)|}\sqrt{\varphi_n'(\xi)}\,\dd \xi\\
&\leq\sqrt{\int_0^{u_{r,n}}|\Lambda''_\delta(\xi)-\Lambda''(\xi)|\dd \xi}\sqrt{\int_0^{u_{r,n}}\varphi_n'(\xi)\dd \xi}\\
&\leq\|\varphi_n(u_{r,n})\|_{L^\infty(Q_T)}^{\frac{1}{2}}\sqrt{\int_0^{u_{r,n}}|\Lambda''_\delta(\xi)-\Lambda''(\xi)|\dd \xi}.
\end{split}
\end{equation}
Since the integrand in the last inequality is dominated by $2p(p-1)|\xi|^{p-2}$ which integrates to $2p|u_{r,n}|^{p-2}u_{r,n}$, we use Lebesgue's dominated convergence theorem to conclude that $\Xi_{n, \delta}\to \Xi_n$ as $\delta\to0^+$. Taking the limit as $\delta\to0^+$ in \eqref{p12normequality}, by using Fatou's lemma on the left-hand side and Lebesgue's dominated convergence theorem ($|\Lambda_\delta(u_0(x))|\leq|u_0(x)|^p$) on the right-hand side, the corollary follows.
\end{proof}

\begin{remark}
Observe that by \eqref{pnormenergyfunction},
\begin{equation*}
\begin{split}
|\Xi_{n, \delta}(u_{r,n})|
&\leq p\|\varphi_n(u_{r,n})\|_{L^\infty(Q_T)}^{\frac{1}{2}}\|u_{r,n}\|_{L^\infty(Q_T)}^{\frac{p-1}{2}}<\infty,
\end{split}
\end{equation*}
and similarly for $\Xi_n(u_{r,n})$. Hence, both are well-defined for all $p\in(1,\infty)$.
\end{remark}

The existence of a distributional solution of \eqref{E}--\eqref{IC}
with finite energy (cf. Theorem \ref{ex_Dsol}) will follow from the following compactness theorem:

\begin{theorem}[Compactness]\label{convapproxtoE}
Assume \textup{(\textup{A}$_\varphi$)}, \textup{(\textup{A}$_\mu$)}, and \textup{(\textup{A}$_{u_0}$)}. Let $\{u_{r,n}\}_{r,n\in\N}$ be a sequence of distributional solutions of \eqref{approxE}--\eqref{approxIC}. Then there exists a subsequence $\{u_{r_j,n_j}\}_{j\in\N}$ and a $u\in C([0,T];L_\textup{loc}^1(\R^N))$ such that
$$
u_{r_j,n_j}\to u \qquad\text{in}\qquad C([0,T];L_\textup{loc}^1(\R^N))\qquad\text{as}\qquad j\to\infty.
$$
Moreover, $u\in L^1(Q_T)\cap L^\infty(Q_T)\cap C([0,T];L_\textup{loc}^1(\R^N))$ is a distributional solution of \eqref{E}--\eqref{IC}.
\end{theorem}

\begin{remark}\label{L1Linftyboundsu}
We have that $\|u\|_{L^1/L^\infty}\leq \|u_0\|_{L^1/L^\infty}$ by Fatou's lemma and Remark \ref{rem:propapproxE} (the limit of a uniformly bounded sequence is uniformly bounded by the same bound).
\end{remark}

\begin{proof}
Observe that the sequence $\{u_{r_j,n_j}\}_{j\in\N}$ enjoy $L^1$-, $L^\infty$-bounds, and time regularity by Remark \ref{rem:propapproxE}, and that these bounds are independent of $j$ (see Section 4 in \cite{DTEnJa15}).

Moreover, for any $\psi\in C_{\textup{c}}^\infty(\R^N)$,
$$
(\Levy^{\mu}-\Levy^{\mu_{r_j}})[\psi](x)=\int_{|z|\leq r_j}\big(\psi(x+z)-\psi(x)-z\cdot D\psi(x)\big)\mu(\dd z),
$$
and hence, $\Levy^{\mu_{r_j}}[\psi]\to \Levy^{\mu}[\psi]$ in $L^1(\R^N)$ as $r_j\to0^+$ by Lebesgue's dominated convergence theorem. We also have,
$$
\sup_{r_j>0}\int_{|z|>0}\min\{|z|^2,1\}\dd\mu_{r_j}(z)\leq \int_{|z|>0}\min\{|z|^2,1\}\dd\mu(z)<\infty,
$$
and $\varphi_{n_j}\to\varphi$ locally uniformly as $n_j\to\infty$ by \eqref{phiapprox}. Thus, we are in the setting of Theorem 2.12 in \cite{DTEnJa15} and the result follows.
\end{proof}

We are now ready to prove Theorems \ref{ex_Dsol}, \ref{e-estim}, and \ref{Lp-estim}. 

\begin{proof}[Proof of Theorem \ref{ex_Dsol} (Existence 1)]
In light of Theorem \ref{convapproxtoE}, it only remains to prove that the limit $u$ is such that $\varphi(u)\in L^2(0,T;E_\mu(\R^N))$. Recall that $\Phi_{n_j}(w)=\int_0^{w}\varphi_{n_j}(\xi)\dd \xi$ and $\Phi(w)=\int_0^{w}\varphi(\xi)\dd
\xi$. Now,
\begin{equation*}
\begin{split}
\left|\int_{\R^N}\Phi(u_0)\dd x-\int_{\R^N}\Phi_{n_j}(u_0)\dd x\right|&\leq \int_{\R^N}|\Phi(u_0)-\Phi_{n_j}(u_0)|\dd x\\
&\leq\int_{\R^N}\int_0^{u_0}|\varphi(\xi)-\varphi_{n_j}(\xi)|\dd \xi\dd x\\
&\leq\|u_0\|_{L^1}\sup_{|\xi|\leq\|u_0\|_{L^\infty}}|\varphi(\xi)-\varphi_{n_j}(\xi)|,
\end{split}
\end{equation*}
and since $\varphi_{n_j}\to\varphi$ locally uniformly, $
\lim_{n_j\to\infty}\int_{\R^N}\Phi_{n_j}(u_0)\dd x=\int_{\R^N}\Phi(u_0)\dd x$.
Observe also that by Theorem \ref{convapproxtoE} (and Remark \ref{rem:propapproxE}), \eqref{phiapprox}, and the proof of Theorem 2.6 in \cite{DTEnJa15}, we can take a further subsequence to get that $\varphi_{n_j}(u_{r_j,n_j})\to \varphi(u)$ a.e. in $Q_T$ as $j\to\infty$.

For any $R\geq r_j>0$, $|\varphi_{n_j}(u_{r_j,n_j})|_{{E}_{\mu_R}}^2\leq |\varphi_{n_j}(u_{r_j,n_j})|_{{E}_{\mu_{r_j}}}^2$, and thus, by the second part of Corollary \ref{energyapproxproblem},
$$
\int_0^\tau|\varphi_{n_j}(u_{r_j,n_j})|_{{E}_{\mu_R}}^2\dd t\leq \int_{\R^N}\Phi_{n_j}(u_0)\dd x.
$$ 
Taking the limit as $j\to\infty$, we obtain, by Fatou's lemma, the above calculations, and the estimate $\Phi(u_0)\leq\|\varphi(u_0)\|_{L^\infty}|u_0|$, that
$$
\int_0^\tau|\varphi(u)|_{{E}_{\mu_R}}^2\dd t\leq \int_{\R^N}\Phi(u_0)\dd x\leq \|\varphi(u_0)\|_{L^\infty}\|u_0\|_{L^1}.
$$ 
Another application of Fatou's lemma, as $R\to0^+$, and the choice $\tau=T$ yield 
$$
|\varphi(u)|_{T,E_{\mu}}\leq\|\varphi(u)\|_{L^\infty(\R^N)}\|u_0\|_{L^1(\R^N)}<\infty.
$$
The proof is complete.
\end{proof}

By Theorem \ref{uniqD}, we know that any subsequence of $\{u_{r,n}\}_{r,n\in\N}$ converges to the same limit, and hence, the whole sequence converges since it is bounded by Remark \ref{rem:propapproxE}. Let us then continue with the proof of the energy and $L^p$-estimates for {\em the} distributional solution of \eqref{E}--\eqref{IC}.

\begin{proof}[Proof of Theorem \ref{e-estim} (Energy inequality)]
By Remark \ref{L1Linftyboundsu},
\begin{equation*}
\begin{split}
&\left|\Phi(u(x,t))-\Phi_n(u_{r,n}(x,t))\right|\\
&\leq\sup_{|\xi|\leq2\|u_0\|_{L^\infty}}|\varphi(\xi)||u(x,t)-u_{r,n}(x,t)|+\|u_0\|_{L^\infty}\sup_{|\xi|\leq\|u_0\|_{L^\infty}}|\varphi(\xi)-\varphi_{n}(\xi)|.
\end{split}
\end{equation*}
Since $\varphi_n\to\varphi$ locally uniformly and we can find a subsequence of $\{u_{r,n}\}_{r,n}$ such that $u_{r_j,n_j}\to u$ a.e. in $Q_T$ as $j\to\infty$ by Theorem \ref{convapproxtoE}, $\Phi_{n_j}(u_{r_j,r_j}(x,t))\to \Phi(u(x,t))$ pointwise a.e. in $Q_T$ as $j\to\infty$. The conclusion then follows by Corollary \ref{energyapproxproblem}, Fatou's lemma, and the proof of Theorem \ref{ex_Dsol}.
\end{proof}

\begin{proof}[Proof of Theorem \ref{Lp-estim} ($L^p$-bounds)]
By Fatou's lemma and Theorem \ref{convapproxtoE}, we can take the limit as $j\to\infty$ (since $u_{r_j,n_j}\to u$ a.e. in $Q_T$ as $j\to\infty$ by considering a further subsequence in Theorem \ref{convapproxtoE}) in the second estimate in Corollary \ref{approxLpbound} to obtain the result. The cases $p=1$ and $p=\infty$ are explained in Remark \ref{L1Linftyboundsu}. 
\end{proof}

\subsection*{Acknowledgments}
\addcontentsline{toc}{section}{Acknowledgments}

E. R. Jakobsen was supported by the Toppforsk (research excellence)
project Waves and Nonlinear Phenomena (WaNP), grant no. 250070 from the Research Council of Norway. F. del Teso was
supported by the FPU grant AP2010-1843 and the grants MTM2011-24696
and MTM2014-52240-P from the Ministry of Education, Culture and
Sports, Spain, the BERC 2014-2017 program from the Basque Government,
and BCAM Severo Ochoa excellence accreditation SEV-2013-0323 from
Spanish Ministry of Economy and Competitiveness (MINECO). We would like to thank Stefano Lisini and Giampiero Palatucci for fruitful discussions on homogeneous fractional Sobolev spaces.

\appendix

\section{Proof of Theorem 2.6 (a)} 
\label{sec:spaceX}

Obviously $X\cap L^2(Q_T)\subset L^2(0,T;E_\lambda(\R^N))\cap L^\infty(Q_T)\cap L^2(Q_T)$, and we must show the opposite inclusion: Any
$$
f\in  L^2(0,T;E_\kernel(\R^N))\cap L^\infty(Q_T)\cap L^2(Q_T)
$$ 
belongs to $X\cap L^2(Q_T)$. To do so, we must prove that $f$ can be suitably approximated by functions in $C_\textup{c}^\infty(\R^N\times[0,T))$. We will now explain how to build such an approximation.

Let $\delta>0$ and $g_\delta:\R^{N+1}\to\R$ be defined by 
$$
g_\delta(x,t):=f(x,t)\mathbf{1}_{[2\delta,T-3\delta]}(t)
$$ 
and mollify $g_\delta$ to get
$$
G_\delta(x,t):=g_\delta\ast_{x,t}\rho_\delta(x,t)=\iint_{\R^{N+1}}g_\delta(y,s)\rho_\delta(x-y,t-s)\,\dd y\dd s
$$
where $\rho_\delta$ is defined by $\rho_\delta(\sigma,\tau):=\frac{1}{\delta^{N+1}}\rho\left(\frac{\sigma}{\delta}, \frac{\tau}{\delta}\right)$ for a fixed $0\leq\rho\in C_\textup{c}^\infty(\R^{N+1})$ satisfying $\text{supp}\,
\rho\subseteq \overline{B}(0,1)\times[-1,1]$, $\rho(\sigma,\tau)=\rho(-\sigma,-\tau)$, and $\iint\rho=1$. Note that $|g_\delta|\leq |f|$ and $G_\delta\in C^\infty(\R^{N+1})$ with support in $\R^N\times[\delta,T-2\delta]$.

\begin{lemma}\label{molliferenergyuniformlybounded}
Assume \textup{(A$_{\kernel0}$)}, \textup{(A$_{\kernel1}$')}, \textup{(A$_{\kernel 2}$)}, and $f\in  L^2(0,T;E_\kernel(\R^N))\cap L^\infty(Q_T)\cap L^2(Q_T)$. 
\vspace{-0.1cm}
\begin{enumerate}[(a)]
\item $G_\delta\in C_0(\R^{N+1})$.
\item $
|G_\delta|_{T,E_\kernel}^2\leq
C|f|_{T,E_\kernel}^2+4\|f\|_{L^2(Q_T)}^2\|\Pi_\kernel\|_{L^\infty(\R^N)}$
for some constant $C\geq0$.
\end{enumerate}
\end{lemma}

\begin{remark}\label{rem:globallyshift-bounded}
If $\kernel$ is {\em globally} shift-bounded, that is, we replace the statement ``$B\subset B(0,1)\setminus\{0\}$'' with ``for all $B\in \R^N\setminus\{0\}$'' in  \textup{(A$_{\kernel1}$')}, then in (b) we get 
$$
|G_\delta|_{T,E_\kernel}^2\leq C|f|_{T,E_\kernel}^2.
$$
In this case, we do not have to assume $\Pi_\kernel\in L^\infty(\R^N)$ and \textup{(A$_{\kernel 2}$)}.
\end{remark}

\begin{proof}
\noindent(a) Since the Fourier transforms of $g_\delta$ and $\rho_\delta$ are both in $L^2(\R^{N+1})$, the properties of the Fourier transform and H\"older's inequality yield 
\begin{equation}\label{Fourierargument}
\mathcal{F}(G_\delta)=\mathcal{F}(g_\delta\ast \rho_\delta)=\mathcal{F}(g_\delta)\mathcal{F}(\rho_\delta)\in L^1(\R^{N+1}).
\end{equation}
The result then follows by the Riemann-Lebesgue lemma which gives that $G_\delta=\mathcal{F}^{-1}(\mathcal{F}(G_\delta))\in C_0(\R^{N+1})$.

\medskip
\noindent(b) The proof is a straightforward adaptation of the proof of Lemma 2.2 in \cite{ScUe12} and the estimate $|g_\delta|\leq|f|$. 
\end{proof}

Next, we recall a useful truncation from \cite{ScUe12}: Let $T_\delta:\R\to\R$ be defined by
$$
T_\delta(x):=\min\left\{\max\left\{-\frac{1}{\delta},x-\min\{\max\{-\delta,x\},\delta\}\right\}, \frac{1}{\delta}\right\}\quad\text{for all}\quad x\in\R.
$$
Observe that for all $x,y\in\R$
\begin{equation}\label{normalcontraction}
|T_\delta(x)|\leq|x|,\quad |T_\delta(x)-T_\delta(y)|\leq|x-y|,\text{ and}\quad T_\delta(x)\to x\text{ as $\delta\to0^+$}.
\end{equation}

We can now define a $C_\textup{c}^\infty$-approximation of $f$:
\begin{align}\label{w-delta}
w_\delta(x,t):=T_\delta[G_\delta]\ast_{x,t}\rho_\delta(x,t).
\end{align}

\begin{lemma}\label{sequenceinL2LinftyE}
Assume \textup{(A$_{\kernel0}$)}, \textup{(A$_{\kernel1}$')}, \textup{(A$_{\kernel 2}$)}, and $f\in  L^2(0,T;E_\kernel(\R^N))\cap L^\infty(Q_T)\cap L^2(Q_T)$. Then:
\begin{enumerate}[(a)]
\item $w_\delta\in C_\textup{c}^\infty(\R^{N+1})$ and $\supp w_\delta\subset \R^N\times[0,T-\delta]$.
\item $\|w_\delta-f\|_{L^2(Q_T)}\to0$ as $\delta\to0^+$.
\item For some $K\geq0$, $\|w_\delta\|_{L^2(Q_T)}+\|w_\delta\|_{L^\infty(Q_T)}+|w_\delta|_{T,E_\kernel}\leq K$ for all $\delta>0$.
\end{enumerate}
\end{lemma}

\begin{remark}\label{L^2L^parg}
If $f\in L^p$ for some $p\in[1,\infty)$, similar arguments show that
 (b) can be replaced by $\|w_\delta-f\|_{L^p(Q_T)}\to0$ as
$\delta\to0^+$.  Moreover, if the measure $\kernel$ is globally
shift-bounded, then we can relax assumption \textup{(A$_{\kernel1}$')} as in Remark
\ref{rem:globallyshift-bounded}, and replace the previous uniform bound of $|w_\delta|_{T,E_\kernel}$ by $C|f|_{T,E_\kernel}$ in (c).
\end{remark}

\begin{proof}
\noindent(a) Since $G_\delta$ vanishes at infinity, $T_\delta[G_\delta]$ has compact support, and therefore $w_\delta:=T_\delta[G_\delta]\ast_{x,t}\rho_\delta\in C_{\textup{c}}^\infty(\R^{N+1})$. Moreover, $\supp T_\delta[G_\delta]\subset \R^N\times[\delta,T-\delta]$ and hence $\supp w_\delta\subset \R^N\times[0,T-\delta]$. As a consequence, $w_\delta\subset C_{\textup{c}}^\infty(\R^N\times[0,T))$.

\medskip
\noindent(b) Note that $|T_\delta[G_\delta]|^2\leq|G_\delta|^2$ by \eqref{normalcontraction}, $\|G_\delta\|_{L^2}\leq\|g_\delta\|_{L^2}$, and $|g_\delta|^2\leq|f|^2$, and thus, all these functions are in $L^2$. Hence,
$$
\|w_\delta-f\|_{L^2(Q_T)}\leq\|w_\delta-G_\delta\|_{L^2(Q_T)}+\|G_\delta-g_\delta\|_{L^2(Q_T)}+\|g_\delta-f\|_{L^2(Q_T)}
$$
and by the properties of mollifiers and \eqref{normalcontraction},
\begin{equation*}
\begin{split}
\|w_\delta-G_\delta\|_{L^2}
\leq&\,\|T_\delta[G_\delta]\ast_{x,t}\rho_\delta-T_\delta[g_\delta]\ast_{x,t}\rho_\delta\|_{L^2}+\|T_\delta[g_\delta]\ast_{x,t}\rho_\delta-G_\delta\|_{L^2}\\
\leq&\,\|T_\delta[G_\delta]-T_\delta[g_\delta]\|_{L^2}+\|T_\delta[g_\delta]-g_\delta\|_{L^2}\\
\leq&\,\|G_\delta-g_\delta\|_{L^2}+\|T_\delta[g_\delta]-g_\delta\|_{L^2}.
\end{split}
\end{equation*}
Finally, we can use Lebesgue's dominated convergence theorem ($|T_\delta[g_\delta]|^2\leq|g_\delta|^2$ by \eqref{normalcontraction} and $|g_\delta|^2\leq|f|^2$) and the properties of mollifiers to conclude.

\medskip
\noindent(c) According to Lemma \ref{molliferenergyuniformlybounded} (b),
\begin{equation*}
\begin{split}
|w_\delta|_{T,E_\kernel}^2\leq C|T_\delta[G_\delta]|_{T,E_\kernel}^2+4\|T_\delta[G_\delta]\|_{L^2(Q_T)}^2\|\Pi_\kernel\|_{L^\infty(\R^N)},
\end{split}
\end{equation*}
and then by \eqref{normalcontraction},
$$
|T_\delta[G_\delta]|_{T,E_\kernel}^2\leq |G_\delta|_{T,E_\kernel}^2\qquad\text{and}\qquad\|T_\delta[G_\delta]\|_{L^2(Q_T)}^2\leq\|G_\delta\|_{L^2(Q_T)}^2.
$$
So, by another application of Lemma \ref{molliferenergyuniformlybounded}, the properties of mollifiers, $|g_\delta|\leq|f|$, and $|g_\delta(x,t)-g_\delta(y,t)|\leq|f(x,t)-f(y,t)|$, we have
\begin{equation*}
\begin{split}
|w_\delta|_{T,E_\kernel}^2
&\leq C\left(|f|_{T,E_\lambda}^2+\|f\|_{L^2(Q_T)}^2\|\Pi_\kernel\|_{L^\infty(\R^N)}\right).
\end{split}
\end{equation*}
Note also that by part (b), $\|w_\delta\|_{L^2(Q_T)}\leq\|f\|_{L^2(Q_T)}$, and moreover, by the properties of mollifiers, \eqref{normalcontraction}, and the definition of $g_\delta$,
\begin{equation*}
\begin{split}
\|w_\delta\|_{L^\infty(Q_T)}\leq\|T_\delta[G_\delta]\|_{L^\infty(\R^{N+1})}\leq\|G_\delta\|_{L^\infty(\R^{N+1})}\leq\|f\|_{L^\infty(Q_T)}.
\end{split}
\end{equation*}
This completes the proof.
\end{proof}

To prove Theorem \ref{prop:characterizationofX} (a), we will define from $\{w_\delta\}_{\delta>0}$ a $C_\textup{c}^\infty$-sequence that converges also in $|\cdot|_{T,E_\kernel}$.

\begin{proof}[Proof of Theorem \ref{prop:characterizationofX} (a)]
  This proof is an adaptation the proof of Theorem 2.4 in
  \cite{ScUe12}. Note that by standard arguments $L^2(Q_T)\cap
  L^2(0,T;E_\kernel(\R^N))$ is a Hilbert space with inner product
  $\langle \cdot,\cdot \rangle_{L^2(Q_T)}+\int_0^T\mathcal{E}_\kernel[\cdot,\cdot]\dd t$. By Lemma
  \ref{sequenceinL2LinftyE} (c) and Banach-Saks' theorem, there is a
  subsequence $\{w_{\delta_k}\}_{k\in\N}$ such that the C\'esaro mean
  of this subsequence converges to some function $\tilde{f}\in
  L^2(Q_T)\cap L^2(0,T;E_\kernel(\R^N))$:
$$
\left\|\frac{1}{n}\sum_{k=1}^nw_{\delta_k}-\tilde{f}\right\|_{L^2(Q_T)}^2+\left|\frac{1}{n}\sum_{k=1}^nw_{\delta_k}-\tilde{f}\right|_{T,E_\kernel}^2\to0\qquad\text{as}\qquad n\to\infty.
$$
Let $\phi_n=\frac{1}{n}\sum_{k=1}^nw_{\delta_k}$. Then $\phi_n\in C_\textup{c}^\infty(\R^N\times[0,T))$ and is uniformly bounded in $L^\infty(Q_T)$ since $w_\delta$ is (cf. Lemma \ref{sequenceinL2LinftyE} (c)). By Banach-Alaoglu's theorem, we can take a further subsequence (also denoted $\phi_n$) such that
$$
\phi_n\overset{*}{\rightharpoonup}\overline{f}\qquad\text{in}\qquad L^\infty(Q_T)\qquad\text{as}\qquad n\to\infty. 
$$

Since $w_\delta\to f$ in $L^2(Q_T)$ as $\delta\to0^+$ by Lemma \ref{sequenceinL2LinftyE} (b), any subsequence and any C\'esaro mean of this subsequence converges to $f\in L^2(Q_T)$. All three notions of convergence implies distributional convergence, and hence, the result follows since by uniqueness of limits, $f=\tilde{f}=\overline{f}$ in $\mathcal{D}'(\R^N\times[0,T))$ and then a.e. 
\end{proof}

\section{On the spaces $\dot H^{\frac\alpha2}(\R^N)$ and $L^2(0,T;\dot H^{\frac\alpha2}(\R^N))$}
\label{sec:homfracSob}

In the first part of this section we prove the equivalence between
three different definitions of the homogeneous Sobolev space $\dot
H^{\frac\alpha2}(\R^N)$ when $N>\alp$. These results are well-known,
but we were unable to find proofs that directly apply to our
setting. Then in the second part, we define the parabolic space
$L^2(0,T;\dot H^{\frac\alpha2}(\R^N))$ and show some of its
properties. Note that we do not define this space as a Bochner space,
but rather as an iterated $L^2$-$\dot H^{\frac\alp2}$ space. Our
discussion heavily relies on  \cite{BePa61}, \cite{BaChDa11},
\cite{ScUe12}, and \cite{PaPi14}.

In the next section we use  these results to prove Theorem \ref{prop:characterizationofX} (b).

\begin{proposition}\label{equivhomfracSob}
Assume $\alpha\in(0,2)$ and $N>\alpha$. Let $f\in \mathcal S'(\R^N)$,
a tempered distribution, $\mathcal{F}\{f\}$ its
Fourier transform, and
\[
|f|_{\Hdot^\frac{\alpha}{2}(\R^N)}^2:=\int_{\R^N}|\xi|^\alpha|\mathcal{F}\{f\}(\xi)|^2\dd \xi<\infty.
\]
The following definitions of $\Hdot^\frac{\alpha}{2}(\R^N)$ are equivalent:
\begin{enumerate}[(a)]
\item 
$\Hdot^\frac{\alpha}{2}_1(\R^N):=\{f\in \mathcal S'(\R^N)\ :\
\mathcal{F}\{f\}\in L_\textup{loc}^1(\R^N)\quad \text{and}\quad  |f|_{\Hdot^\frac{\alpha}{2}(\R^N)}<\infty\},$
\item $\Hdot^\frac{\alpha}{2}_2(\R^N):=\overline{C_\textup{c}^\infty(\R^N)}^{|\cdot|_{\Hdot^\frac{\alpha}{2}(\R^N)}}$, and
\item $\Hdot^\frac{\alpha}{2}_3(\R^N):=\{f\in L^{\frac{2N}{N-\alpha}}(\R^N) \,:\, |f|_{\Hdot^\frac{\alpha}{2}(\R^N)}<\infty\}.$
\end{enumerate}
\end{proposition}

\begin{proof}
{\sc 1)}\quad By Propositions 1.34 and 1.37 and Theorem 1.38 in \cite{BaChDa11},  $\Hdot_1^\frac{\alpha}{2}(\R^N)$ is a Hilbert space with the norm
$$
|f|_{\Hdot^\frac{\alpha}{2}(\R^N)}=\int_{\R^N}|\xi|^\alpha|\mathcal{F}\{f\}(\xi)|^2\dd \xi=c_{N,\alpha}\int_{\R^N}\int_{|z|>0}\frac{|f(x+z)-f(x)|^2}{|z|^{N+\alpha}}\dd z\dd x,
$$
and $\Hdot_1^\frac{\alpha}{2}(\R^N)$ is continuously embedded in $L^{\frac{2N}{N-\alpha}}(\R^N)$ with
\begin{equation}\label{emb}
\|f\|_{L^{\frac{2N}{N-\alpha}}(\R^N)}\leq C |f|_{\Hdot^\frac{\alpha}{2}(\R^N)}.
\end{equation}
We also note that $\frac{2N}{N-\alpha}\in(2,\infty)$ as long as $N>\alpha$. 

\smallskip
\noindent  {\sc 2)}\quad $\Hdot^\frac{\alpha}{2}_2(\R^N) \subset \Hdot^\frac{\alpha}{2}_1(\R^N)$:  For any $f\in \Hdot^\frac{\alpha}{2}_2(\R^N)$,  it is clear that $f$ is a tempered distribution and $|f|_{\Hdot^\frac{\alpha}{2}}(\R^N)<\infty$. Furthermore, $\mathcal{F}\{f\}\in L^1_{\textup{loc}}(\R^N)$ since  for any compact $K\subset\R^N$, we use H\"{o}lder's inequality to get
\begin{equation}\label{L1locF}
\int_K|\mathcal{F}\{f\}|\dd \xi\leq\left(\int_{\R^N}|\xi|^\alpha|\mathcal{F}\{f\}|^2\dd \xi\right)^{\frac{1}{2}}\left(\int_{K}|\xi|^{-\alpha}\dd \xi\right)^{\frac{1}{2}}<\infty.
\end{equation}

\noindent  {\sc 3)}\quad $\Hdot^\frac{\alpha}{2}_1(\R^N) \subset \Hdot^\frac{\alpha}{2}_2(\R^N)$: Due to Remark \ref{L^2L^parg}, we can proceed as in Section \ref{sec:spaceX} (or Theorem 2.4 in \cite{ScUe12}): 
For any $f\in \Hdot^{\frac{\alpha}{2}}_1(\R^N)$, we construct an $C_\textup{c}^\infty$-approximation $w_\delta$ satisfying 
\begin{equation}\label{prophomfracSobapprox}
\|w_\delta-f\|_{L^\frac{2N}{N-\alpha}(\R^N)}\overset{\delta\to0^+}{\longrightarrow}0\qquad\text{and}\qquad|w_\delta|_{\Hdot^\frac{\alpha}{2}(\R^N)}\leq C|f|_{\Hdot^\frac{\alpha}{2}(\R^N)}.
\end{equation}
Hence since $\Hdot^\frac{\alpha}{2}_1(\R^N)$ is a Hilbert space, Banach-Saks' theorem ensures the existence of a subsequence $\{w_{\delta_k}\}_{k\in\N}$ and $\tilde{f}\in \Hdot^\frac{\alpha}{2}_1(\R^N)$ such that
$$
\left|\frac{1}{n}\sum_{k=1}^nw_{\delta_k}-\tilde{f}\right|_{\Hdot^\frac{\alpha}{2}(\R^N)}^2\to0\qquad\text{as}\qquad n\to\infty.
$$
By \eqref{emb}, these C\'esaro means converge to $\tilde f$ in
$L^{\frac{2N}{N-\alpha}}$. But by \eqref{prophomfracSobapprox}, they
also converge to $f$ in $L^{\frac{2N}{N-\alpha}}$, and hence $f=\tilde{f}$ a.e.

\smallskip
\noindent  {\sc 4)}\quad $\Hdot^\frac{\alpha}{2}_3(\R^N) \subset \Hdot^\frac{\alpha}{2}_1(\R^N)$: Since $f\in L^{\frac{2N}{N-\alpha}}(\R^N)$, it is a tempered distribution, and $\mathcal{F}\{f\}\in L^1_{\textup{loc}}(\R^N)$ by \eqref{L1locF}.

\smallskip
\noindent  {\sc 5)} $\Hdot^\frac{\alpha}{2}_1(\R^N) \subset \Hdot^\frac{\alpha}{2}_3(\R^N)$: This is just a consequence of \eqref{emb}.
\end{proof}

We now define and analyze the parabolic space
$L^2(0,T;\Hdot^\frac{\alpha}{2}(\R^N))$. In the
proof we will use the following iterated $L^p$-space \cite{BePa61}:
$$L^2(0,T;L^{q}(\R^N))=\Big\{f:\R^N\times(0,T)\to\R\
\text{ measurable}  \mid \int_0^T\|f(\cdot,t)\|^2_{L^q}\dd t<\infty\Big\},$$
for some $q\in(1,\infty)$. Note that this space is not a priori a Bochner space.

\begin{lemma}\label{parahomfracSob}
Let $\alpha\in(0,2)$, $N>\alpha$, and $\mu_\alpha(\dd z):=\frac{c_{N,\alpha}\dd z}{|z|^{N+\alpha}}$. Then the space
\begin{equation*}
\begin{split}
  L^2(0,T;\Hdot^\frac{\alpha}{2}(\R^N))&
  :=L^2(0,T;L^{\frac{2N}{N-\alpha}}(\R^N))\cap L^2(0,T;E_{\mu_\alpha}(\R^N))
\end{split}
\end{equation*}
is a Hilbert space with inner product 
\begin{equation*}
\begin{split}
\langle\psi,\phi\rangle 
&:=\int_0^T\int_{\R^N}|\xi|^\alpha\mathcal{F}\{\psi(\cdot,t)\}(\xi)\overline{\mathcal{F}\{\phi(\cdot,t)\}(\xi)}\dd \xi\dd t=c_{N,\alp}\int_0^T\mathcal{E}_{\mu_\alpha}[\psi,\phi]\dd t.
\end{split}
\end{equation*}
Moreover, $L^2(0,T;\Hdot^\frac{\alpha}{2}(\R^N))=\overline{C_\textup{c}^\infty(\R^N\times[0,T))}^{\sqrt{\langle\cdot,\cdot\rangle}}$.
\end{lemma}

\begin{proof}
{\sc 1)}\quad Embedding. Since $f\in L^2(0,T;\Hdot^\frac{\alpha}{2}(\R^N))$, we
have as a consequence of properties of iterated $L^p$-spaces \cite{BePa61} and Fubini's theorem that
$$
\|f(\cdot,t)\|_{L^{\frac{2N}{N-\alpha}}(\R^N)}\in L^2(0,T)\qquad\text{and}\qquad |f(\cdot,t)|_{\Hdot^\frac{\alpha}{2}(\R^N)}\in L^2(0,T).
$$
It follows that for a.e. $t\in(0,T)$, $f(\cdot,t)\in L^{\frac{2N}{N-\alpha}}(\R^N)\cap
E_{\mu_\alpha}(\R^N)$, and then $f(\cdot,t)\in
\Hdot^\frac{\alpha}{2}_1(\R^N)$  by Proposition
\ref{equivhomfracSob} (c). By \eqref{emb}, we can then conclude that
\begin{equation}\label{paremb}
\int_0^T\|f(\cdot,t)\|_{L^{\frac{2N}{N-\alpha}}(\R^N)}^2\dd t\leq C\int_0^T|f(\cdot,t)|_{{\Hdot^\frac{\alpha}{2}}(\R^N)}^2\dd t.
\end{equation}
\smallskip
\noindent  {\sc 2)}\quad Inner product space. Obviously $\langle\cdot,\cdot\rangle=\int_0^T|\cdot|_{{\Hdot^\frac{\alpha}{2}}(\R^N)}^2\dd
t=|\cdot|_{T,E_{\mu_\alp}}^2$ defines (the square of) a seminorm. By
\eqref{paremb}, 
$\int_0^T|f(\cdot,t)|_{{\Hdot^\frac{\alpha}{2}}(\R^N)}^2\dd t=0$
implies $f=0$ a.e. in $Q_T$, and hence the seminorm is a full
norm. Now it is easy to check that the space is an inner product space.

\medskip
\noindent  {\sc 3)}\quad Completeness. Let $\{f_n\}_{n\in\N}$ be a
Cauchy sequence in $L^2(0,T;\Hdot^\frac{\alpha}{2}(\R^N))$. By
definition and  \eqref{paremb}, it follows that
$$
\int_0^T\|f_n(\cdot,t)-f_m(\cdot,t)\|_{L^{\frac{2N}{N-\alpha}}(\R^N)}^2\dd
t\leq\int_0^T|f_n(\cdot,t)-f_m(\cdot,t)|_{{\Hdot^\frac{\alpha}{2}}(\R^N)}^2\dd
t\to 0
$$
as $n,m\to\infty$. Hence the sequence is Cauchy also in
$L^2(0,T;L^{\frac{2N}{N-\alpha}}(\R^N))$. By \cite{BePa61}, this space
is complete, and sequences converging in norm contain pointwise
a.e. converging subsequences.
Therefore there is 
$f\in L^2(0,T;L^{\frac{2N}{N-\alpha}}(\R^N))$ such that
$$
\int_0^T\|f_n(\cdot,t)-f(\cdot,t)\|_{L^{\frac{2N}{N-\alpha}}(\R^N)}\to0\qquad\text{as}\qquad n\to\infty, 
$$
and a subsequence $f_{n_k}\to f$ a.e. in $Q_T$ as
$n\to\infty$. By Fatou's lemma,  
\begin{equation*}
\begin{split}
\int_0^T|f_{m}(\cdot,t)-f(\cdot,t)|_{{\Hdot^\frac{\alpha}{2}}(\R^N)}^2\dd t&= \int_0^T|f_m(\cdot,t)-\lim_{k\to\infty}f_{n_k}(\cdot,t)|_{{\Hdot^\frac{\alpha}{2}}(\R^N)}^2\dd t\\
&\leq  \liminf_{k\to\infty}\int_0^T|f_m(\cdot,t)-f_{n_k}(\cdot,t)|_{{\Hdot^\frac{\alpha}{2}}(\R^N)}^2\dd t,
\end{split}
\end{equation*}
which goes to zero as $m\to\infty$.  Hence $f_n\to f$ in
$L^2(0,T;\Hdot^\frac{\alpha}{2}(\R^N))$, and then by the triangle
inequality $f\in L^2(0,T;\Hdot^\frac{\alpha}{2}(\R^N))$. 

\medskip

\noindent  {\sc 4)}\quad Density. 
For any $f\in L^2(0,T;\Hdot^\frac{\alpha}{2}(\R^N))$, the
$C_\textup{c}^\infty(\R^N\times[0,T))$ functions $w_\delta$ defined in
\eqref{w-delta} in Appendix \ref{sec:spaceX} satisfy 
\begin{equation}\label{parprophomfracSobapprox}
\begin{split}
\|w_\delta-f\|_{L^2(0,T;L^\frac{2N}{N-\alpha}(\R^N))}\overset{\delta\to0^+}{\longrightarrow}0\qquad\text{and}\qquad
|w_\delta|_{T,E_{\mu_\alpha}}\leq C|f|_{T,E_{\mu_\alpha}}.
\end{split}
\end{equation}
This follows from Remark \ref{L^2L^parg}, and the fact that the iterated
$L^p$-spaces have similar properties as the usual $L^p$-spaces with
respect to mollifications (by \cite{BePa61},
inequalities for convolutions, continuity of translations, dominated
convergence etc. are similar). Since
$L^2(0,T;\Hdot^\frac{\alpha}{2}(\R^N))$ is a Hilbert space,
Banach-Saks' theorem implies there is a subsequence $\{w_{\delta_k}\}_{k\in\N}$ and
$\tilde{f}\in L^2(0,T;\Hdot^\frac{\alpha}{2}(\R^N))$ such that
$$
\left|\frac{1}{n}\sum_{k=1}^nw_{\delta_k}-\tilde{f}\right|_{T,E_{\mu_\alpha}}^2\to0\qquad\text{as}\qquad n\to\infty.
$$
By \eqref{paremb}, these C\'esaro means converge to $\tilde f$ in
$L^2(0,T;L^\frac{2N}{N-\alpha}(\R^N))$. But by \eqref{parprophomfracSobapprox}, they
also converge to $f$ in this space, and hence $f=\tilde{f}$ a.e.
\end{proof}


\section{Proof of Theorem 2.6 (b)}
\label{sec:proofcomparablefracLap}

{\sc 1)}\quad By \textup{(A$_{\lambda 1}$'')}, the measure
$\kernel$ is globally shift-bounded and the (semi)norms on $E_\kernel$ and
$\Hdot^\frac{\alpha}{2}$ are comparable:
$ m|f|_{\Hdot^\frac{\alpha}{2}}\leq|f|_{E_\kernel}\leq M|f|_{\Hdot^\frac{\alpha}{2}}$. The latter gives
$$
E_\kernel(\R^N)=\{f\textup{ is measurable}:
|f|_{\Hdot^\frac{\alpha}{2}}<\infty\}=E_{\mu_\alpha}(\R^N).
$$ 

\smallskip
\noindent{\sc 2)}\quad For $N\leq \alpha$, the fractional Laplacian
(and hence also $A^\kernel$) is recurrent. In this case an easy
modification of the proof of Theorem 3.1 in \cite{ScUe12} yields
$$
X=L^\infty(Q_T)\cap L^2(0,T;E_{\mu_\alpha}(\R^N)).
$$
See the discussion on recurrence in Section \ref{sec:rem} for more details.

\medskip
\noindent{\sc 3)}\quad For $N>\alpha$, we always have $X\subset
L^\infty(Q_T)\cap L^2(0,T;E_{\mu_\alpha}(\R^N))$. To prove the reverse
inclusion, we must show that $C_\textup{c}^\infty$ is dense in
$L^\infty(Q_T)\cap L^2(0,T;E_{\mu_\alpha}(\R^N))$. This will follow
from  Lemma \ref{parahomfracSob} if we can show that 
$$L^\infty(Q_T)\cap L^2(0,T;E_{\mu_\alpha}(\R^N))\subset L^2(0,T;\dot H^{\frac\alp2}(\R^N)). $$
To prove this, we must show that any $g\in L^\infty(Q_T)\cap
L^2(0,T;E_{\mu_\alpha}(\R^N))$ also belongs to $L^2(0,T;L^{\frac{2N}{N-\alp}}(\R^N))$. 

Let
$h(t):=\|g(\cdot,t)\|_{L^{\frac{2N}{N-\alpha}}(\R^N)}^2$ for
$t\in(0,T)$. By Section 252P in \cite{Fre01} and measurability of $g$
on $\R^N\times(0,T)$, $h$ is a measurable function on
  $(0,T)$. We need to prove that $h$ belongs to
$L^1(0,T)$. As  a consequence of Fubini's theorem,  
$$
\|g(\cdot,t)\|_{L^\infty(\R^N)}\in L^\infty(0,T)\qquad\text{and}\qquad |g(\cdot,t)|_{\Hdot^\frac{\alpha}{2}(\R^N)}\in L^2(0,T),
$$
and then $g(\cdot,t)\in L^\infty(\R^N)\cap E_{\mu_\alpha}(\R^N)$
for a.e. $t\in(0,T)$. Hence for such $t$, $g(\cdot,t)$ is a tempered
distribution and belongs to $\Hdot^\frac{\alpha}{2}_1(\R^N)$ by the
argument of Step 4) in the proof of Proposition
\ref{equivhomfracSob}. Then by the embedding 
\eqref{emb}, $h(t)\leq
C|g(\cdot,t)|_{\Hdot^\frac{\alpha}{2}(\R^N)}^2$ for
a.e. $t$, and hence since
$|g(\cdot,t)|_{\Hdot^\frac{\alpha}{2}(\R^N)}^2\in L^1(0,T)$, it follows
that $h\in L^1(0,T)$ and $g\in L^2(0,T;L^{\frac{2N}{N-\alp}}(\R^N))$. The
proof is complete.


\section{Proof of Lemma 4.3}
\label{sec:technicalresult}

{\sc 1)}\quad Assume $1<p,q<\infty$. Consider $f_n(x):=(f\ast\omega_n)(x)$ and $g_n(x):=(g\ast\omega_n)(x)$ for $\omega_n$ defined by \eqref{mollifierspace}. By a direct computation,

\begin{equation*}
\begin{split}
&(f_ng_n)(x+z)-(f_ng_n)(x)-z\cdot D(f_ng_n)(x)\\
&=f_n(x)\left(g_n(x+z)-g_n(x)-z\cdot Dg_n(x)\right)\\
&\quad+g_n(x)\left(f_n(x+z)-f_n(x)-z\cdot Df_n(x)\right)\\
&\quad+(f_n(x+z)-f_n(x))(g_n(x+z)-g_n(x))\\
\end{split}
\end{equation*}
Integrate the above equality against $\nu(\dd z)\dd x$ to get
\begin{equation*}
\begin{split}
&\int_{\R^N}\Levy^{\nu}[f_ng_n](x)\dd x\\
&=\int_{\R^N}f_n(x)\Levy^{\nu}[g_n](x)\dd x+\int_{\R^N}g_n(x)\Levy^{\nu}[f_n](x)\dd x+2\mathcal{E}_\nu[f_n,g_n].
\end{split}
\end{equation*}
The three terms on the right-hand side are well-defined by H\"older's
inequality since the measure $\nu$ is finite. By Fubini's theorem, $\int_{\R^N}\Levy^{\nu}[f_ng_n](x)\dd x=0$
and thus, we obtain
\begin{equation}\label{approxsimpleStroock-Var}
\begin{split}
0=&\int_{\R^N}f_n(x)\Levy^{\nu}[g_n](x)\dd x+\int_{\R^N}g_n(x)\Levy^{\nu}[f_n](x)\dd x+2\mathcal{E}_\nu[f_n,g_n].
\end{split}
\end{equation}
By standard estimates for mollifiers, Tonelli's lemma, and H\"older's inequality,
\begin{equation*}
\begin{split}
&\left|\int_{\R^N}f_n\Levy^\nu[g_n]\dd x - \int_{\R^N}f\Levy^\nu[g]\dd x\right|\\
&\leq 2\nu(\R^N)\left(\|g\|_{L^q(\R^N)}\|f_n-f\|_{L^p(\R^N)}+\|f\|_{L^p(\R^N)}\|g_n-g\|_{L^q(\R^N)}\right),
\end{split}
\end{equation*}
and
\begin{equation*}
\begin{split}
&\big|2\mathcal{E}_\nu[f_n,g_n]-2\mathcal{E}_\nu[f,g]\big|\\
&\leq 4\nu(\R^N)\left(\|g\|_{L^q(\R^N)}\|f_n-f\|_{L^p(\R^N)}+\|f\|_{L^p(\R^N)}\|g_n-g\|_{L^q(\R^N)}\right).
\end{split}
\end{equation*}
Note that a similar argument holds for $\int_{\R^N}g_n\Levy^\nu[f_n]\dd x$. 
Taking the limit as $n\to\infty$ and using the properties of mollifiers, we obtain \eqref{approxsimpleStroock-Var} for $f,g$ replacing $f_n,g_n$ respectively. Since $\Levy^\nu$ is symmetric, we obtain
$$
\int_{\R^N}g(x)\Levy^{\nu}[f](x)\dd x=-\mathcal{E}_\nu[f,g].
$$

\smallskip
\noindent  {\sc 2)}\quad Assume $p=1, q=\infty$. Again we mollify, $f_n(x):=(f\ast\omega_n)(x)$ and $g_m(x):=(g\ast\omega_m)(x)$, and we obtain \eqref{approxsimpleStroock-Var} as above. 
We deduce (almost as before) that
\begin{equation*}
\begin{split}
&\left|\int_{\R^N}f_n\Levy^\nu[g_m]\dd x - \int_{\R^N}f\Levy^\nu[g]\dd x\right|\\
&\leq 2\nu(\R^N)\|g\|_{L^\infty(\R^N)}\|f_n-f\|_{L^1(\R^N)}\\
&\quad+\int_{\R^N}\int_{\R^N}\left|f_n(x)\right|\left|(g_m(x+z)-g(x+z))-(g_m(x)-g(x))\right| \,\nu(\dd z)\dd x,
\end{split}
\end{equation*}
and
\begin{equation*}
\begin{split}
\big|2\mathcal{E}_\nu[f_n,g_m]-2\mathcal{E}_\nu[f,g]\big|\leq &\int_{\R^N}\int_{\R^N}\left|f_n(x+z)-f_n(x)\right|\times\\
&\times\left|(g_m(x+z)-g(x+z))-(g_m(x)-g(x))\right| \nu(\dd z)\dd x\\
&+4\nu(\R^N)\|g\|_{L^\infty(\R^N)}\|f_n-f\|_{L^1(\R^N)}\\
\end{split}
\end{equation*}
Note that $\left|(g_m(x+z)-g(x+z))-(g_m(x)-g(x))\right|\leq4\|g\|_{L^\infty(\R^N)}$ and \linebreak$|f_n(x)|\in L^1(\R^N)$. Hence, for fixed $n$, we may send $m\to\infty$ by Lebesgue's dominated convergence theorem to obtain \eqref{approxsimpleStroock-Var} for $f_n,g$. Then we send $n\to\infty$ to obtain the same for $f,g$. Again, we use the symmetry of $\Levy^\nu$ to complete the proof.



\begin{thebibliography}{10}
\addcontentsline{toc}{section}{References}

\bibitem{Ali07}
{\sc N.~Alibaud.}
\newblock Entropy formulation for fractal conservation laws.
\newblock {\em J.~Evol.~Equ.}, 7(1):145--175, 2007.

\bibitem{AlAn10}
{\sc N.~Alibaud and B.~Andreianov.}
\newblock Non-uniqueness of weak solutions for the fractal Burgers equation.
\newblock {\em Ann. Inst. H. Poincar\'e
Anal. Non Lin\'eaire}, 27(4):997--1016, 2010.

\bibitem{AMRT10}
{\sc F. Andreu-Vaillo, J. M. Mazon, J. D. Rossi and J. J. Toledo-Melero.} 
\newblock {\em Nonlocal Diffusion Problems}
\newblock Math. Surveys Monogr., 165, AMS, Rhode Island, 2010.

\bibitem{App09} 
{\sc D.~Applebaum.}
\newblock \emph{L\'evy processes and Stochastic Calculus.}
\newblock Cambridge Studies in Advanced Mathematics, 116, Cambridge University Press, Cambridge, 2009.

\bibitem{BaBaChKa09}
{\sc M.~T.~Barlow, R.~F.~Bass, Z.-Q.~Chen and M.~Kassmann.}
\newblock Non-local Dirichlet forms and symmetric jump processes. 
\newblock {\em Trans. Amer. Math. Soc.}, 361(4):1963--1999, 2009.

\bibitem{BaPeSoVa14}
{\sc B. Barrios, I. Peral, F. Soria and E. Valdinoci}.
\newblock A Widder's type theorem for the heat equation with nonlocal diffusion.
\newblock \emph{Arch. Ration. Mech. Anal.}, 213(2):629--650, 2014.

\bibitem{BaChDa11}
{\sc H.~Bahouri, J.-Y.~Chemin and R.~Danchin.}
\newblock {\em Fourier Analysis and Nonlinear Partial Differential Equations.}
\newblock Grundlehren der mathematischen Wissenschaften, 343, Springer-Verlag, Berlin-Heidelberg, 2001.


\bibitem{BePa61}
{\sc A.~Benedek and R.~Panzone.}
\newblock The space $L^p$, with mixed norm.
\newblock {\em Duke Math. J.}, 28(3):301--324, 1961.


\bibitem{BiImKa15}
{\sc P.~Biler, C.~Imbert and G.~Karch.}
\newblock The nonlocal porous medium equation: {B}arenblatt profiles and other weak solutions.
\newblock \emph{Arch. Ration. Mech. Anal.}, 215:497--529, 2015.

\bibitem{BiKaMo10}
{\sc P.~Biler, G. Karch and R. Monneau.}
\newblock Nonlinear diffusion of dislocation density and self-similar solutions. 
\newblock \emph{Comm. Math. Phys.}, 294(1):145--168, 2010.

\bibitem{BoSeVa16}
{\sc M. Bonforte, A. Segatti and J. L. V\'azquez.}
\newblock Non-existence and instantaneous extinction of solutions for singular nonlinear fractional diffusion equations.
\newblock \emph{Calc. Var. Partial Differential Equations}, 55(3):55--68, 2016.

\bibitem{BoSiVa14}
{\sc M. Bonforte, Y. Sire and J. L. Vazquez.}
\newblock Existence, Uniqueness and Asymptotic behaviour for fractional porous medium equations on bounded domains.
\newblock \emph{Discrete Contin. Dyn. Syst.}, 35(12):5725--5767, 2015.

\bibitem{BoSiVa16}
{\sc M. Bonforte, Y. Sire and J. L. V\'azquez.}
\newblock  Optimal Existence and Uniqueness Theory for the Fractional Heat Equation.
\newblock  {\em Nonlinear Anal. TMA.}, 153:142--168, 2017.  

\bibitem{BoVa14}
{\sc M. Bonforte and J. L. V\'azquez.} 
\newblock Quantitative Local and Global A Priori Estimates for Fractional Nonlinear Diffusion Equations. 
\newblock \emph{Adv. Math.}, 250:242--284, 2014.

\bibitem{BoVa15}
{\sc M. Bonforte and J. L. V\'azquez.}
\newblock A Priori Estimates for Fractional Nonlinear Degenerate Diffusion Equations on bounded domains.
\newblock \emph{Arch. Ration. Mech. Anal.}, 218(1):317--362, 2015.


\bibitem{BoVa16}
{\sc M. Bonforte and J. L. V\'azquez.}
\newblock Fractional Nonlinear Degenerate Diffusion Equations on Bounded Domains Part I. Existence, Uniqueness and Upper Bounds.
\newblock \emph{Nonlinear Anal. TMA.}, 131:363--398,  2016.


\bibitem{BrPa15}
{\sc C.~Br\"andle and A.~de~Pablo.}
\newblock Nonlocal heat equations: decay estimates and Nash inequalities.
\newblock Preprint, arXiv:1312.4661v4 [math.AP], 2015.

\bibitem{BrCr79}
{\sc H. Br\'ezis and M. G. Crandall.}
\newblock Uniqueness of solutions of the initial--value problem for $u_t-\triangle \varphi(u)=0.$
\newblock {\it J. Math. Pures Appl.}, 58(2):153--163, 1979.

 
 \bibitem{CaVa11}
 {\sc L.~Caffarelli and J.~L. Vazquez.}
 \newblock Nonlinear porous medium flow with fractional potential pressure.
 \newblock \emph{Arch. Ration. Mech. Anal.}, 202(2):537--565, 2011.


\bibitem{CiJa11}
{\sc S.~Cifani and E.~R. Jakobsen.}
\newblock Entropy formulation for degenerate fractional order convection-diffusion equations.
\newblock \newblock \emph{Ann. Inst. H. Poincar\'e
Anal. Non Lin\'eaire}, 28(3):413--441, 2011.


\bibitem{EnJa14}
{\sc J.~Endal and E.~R.~Jakobsen.}
\newblock $L^1$ Contraction for Bounded (Nonintegrable) Solutions of Degenerate Parabolic Equations.
\newblock {\em SIAM J. Math. Anal.}, 46(6):3957--3982, 2014.


\bibitem{DTEnJa15}
{\sc F.~del Teso, J.~Endal and E.~R.~Jakobsen.}
\newblock Uniqueness and properties of distributional solutions of nonlocal equations of porous medium type.
\newblock \emph{Adv. Math.}, 305:78--143, 2017.

\bibitem{DTEnJa16}
{\sc F.~del Teso, J.~Endal and E.~R.~Jakobsen.}
\newblock Numerical analysis and methods for distributional solutions of nonlocal (and local) equations of porous medium type.
\newblock In preparation, 2016.

\bibitem{PaQuRo16}
{\sc A.~de Pablo, F.~Quir\'os and A.~Rodr\'iguez.}
\newblock Nonlocal filtration equations with rough kernels.
\newblock \emph{Nonlinear Anal. TMA}, 137:402--425, 2016.

\bibitem{PaQuRoVa11}
{\sc A.~de Pablo, F.~Quir\'os, A.~Rodr\'iguez and J.~L.~V\'azquez.}
\newblock A fractional porous medium equation.
\newblock \emph{Adv. Math.}, 226(2):1378--1409, 2011.

\bibitem{PaQuRoVa12}
{\sc A.~de~Pablo, F.~Quir\'os, A.~Rodr\'iguez and J.~L.~V\'azquez.}
\newblock A general fractional porous medium equation.
\newblock \emph{Comm. Pure Appl. Math.}, 65(9):1242--1284, 2012.

\bibitem{PaQuRoVa14}
{\sc A.~de Pablo, F.~Quir\'os, A.~Rodr\'iguez and J.~L.~V\'azquez.}
\newblock Classical solutions for a logarithmic fractional diffusion equation.
\newblock \emph{J. Math. Pures Appl.}, 101(6):901--924, 2014. 



\bibitem{Fre01}
{\sc D.~H.~Fremlin.}
\newblock {\em Measure theory. {V}ol. 2: Broad Foundations.}
\newblock Torres Fremlin, Colchester, 2001.


\bibitem{FuOsTa94}
{\sc M.~Fukushima, Y.~Oshima, and M.~Takeda.}
\newblock \emph{Dirichlet Forms and Symmetric Markov Processes.}
\newblock Studies in Mathematics, 19, De Gruyter, Berlin, 1994.

\bibitem{GrMuPu15}
{\sc L.~Grillo, M. Muratori and F. Punzo.}
\newblock Fractional porous media equations: existence and uniqueness
of weak solutions with measure data.
\newblock {\em Calc. Var. Partial Differential Equations}, 54(3):3303--3335, 2015.

\bibitem{KaSc14}
{\sc M.~Kassmann and R.~W.~Schwab.} 
\newblock Regularity results for nonlocal parabolic equations.
\newblock {\em Riv. Mat. Univ. Parma}, 5(1):183--212, 2014.


\bibitem{OlKaCz58}
{\sc O.~A.~Ole\u{i}nik, A.~S.~Kala\v{s}nikov and Y.-I.~\v{C}\v{z}ou.}
\newblock The Cauchy problem and boundary problems for equations of
the type of non-stationary filtration. 
\newblock (In Russian) {\em Izv. Akad. Nauk SSSR. Ser. Mat.}, 22:667--704, 1958. 


\bibitem{PaPi14}
{\sc G.~Palatucci and A.~Pisante.}
\newblock Improved Sobolev embeddings, profile decomposition, and concentration-compactness for fractional Sobolev spaces.
\newblock \emph{Calc.Var.}, 50:799--829, 2014.


\bibitem{Sat99}
{\sc K.~Sato.}
\newblock {\em L\'{e}vy processes and infinitely divisible distributions.}
\newblock Cambridge Studies in Advanced Mathematics, 68, Cambridge University Press, Cambridge, 1999.

\bibitem{ScUe07}
{\sc R.~L.~Schilling and T.~Uemura.}
\newblock On the Feller property of Dirichlet forms generated by pseudo differential operators.
\newblock \emph{T\^ohoku Math. J.}, 59:401--422, 2007.

\bibitem{ScUe12}
{\sc R.~L.~Schilling and T.~Uemura.}
\newblock On the Structure of the Domain of a Symmetric Jump-type Dirichlet Form.
\newblock \emph{Publ. RIMS Kyoto Univ.}, 48:1--20, 2012.

\bibitem{Sil14}
{\sc L.~Silvestre.}
\newblock Regularity estimates for parabolic integro-differential equations and
applications. 
\newblock In: {\em Proceedings of the International Congress of Mathematicians}, vol. III, 873--894, Kyung Moon SA, Seoul, 2014.


\bibitem{StanTesoVazCRAS}
{\sc D.~Stan, F.~ del Teso and J.~L.~V\'azquez.}
\newblock Finite and infinite speed of propagation for porous medium equations with fractional pressure.
\newblock \emph{C. R. Math. Acad. Sci. Paris}, 119:62--73, 2014.

\bibitem{StTeVa15}
{\sc D.~Stan, F.~ del Teso and J.~L.~V\'azquez.}
\newblock Transformations of self-similar solutions for porous medium equations of fractional type.
\newblock {\em Nonlinear Anal. TMA}, 119:62--73, 2015. 

\bibitem{StTeVa16}
{\sc D.~ Stan, F.~ del Teso and  J.~ L.~V{\'a}zquez}. 
\newblock Finite and infinite speed of propagation for porous medium equations with nonlocal pressure.
\newblock \emph{J. Differential Equations}, 260(2): 1154--1199, 2016.


\bibitem{SV79} 
{\sc D. W. Stroock and S. R. S. Varadhan.}
\newblock {\em Multidimensional diffusion processes.}
\newblock Grundlehren der mathematischen Wissenschaften, 233, 
Springer-Verlag, Berlin-New York, 1979.

\bibitem{Vaz06}
{\sc J.~L.~V\'azquez.}
\newblock \emph{Smoothing and decay estimates for nonlinear diffusion equations. Equations of porous medium type.} 
\newblock Oxford Lecture Series in Mathematics and its Applications, 33, Oxford University Press, Oxford, 2006.

\bibitem{Vaz07}
{\sc J.~L.~V\'azquez.}
\newblock \emph{The porous medium equation. Mathematical theory.}
\newblock Oxford Math. Monogr., The Clarendon Press, Oxford University Press, Oxford, 2007.


\bibitem{Zie89}
{\sc W.~P.~Ziemer.}
\newblock {\em Weakly Differentiable functions: Sobolev Spaces and Functions of Bounded Variation.}
\newblock Graduate texts in mathematics, 120, Springer-Verlag, New York, 1989.

\end{thebibliography}
\end{document}